\documentclass[reqno]{amsart}
\usepackage{amsmath,amsthm,amscd,amssymb,amsfonts, amsbsy}
\usepackage{latexsym, color, enumerate}
\usepackage{pxfonts}
\usepackage{mathrsfs}
\usepackage{marginnote}
\usepackage{todonotes}
\usepackage{hyperref}

\usepackage{tikz}%draw picture
\usetikzlibrary{snakes}%zigzagline
\usepackage{multicol}%multicolumn

\numberwithin{equation}{section}

\theoremstyle{plain}
\newtheorem{theorem}{Theorem}[section]
\newtheorem{lemma}[theorem]{Lemma}

\newtheorem{proposition}[theorem]{Proposition}

\theoremstyle{definition}
\newtheorem{definition}[theorem]{Definition}

\newtheorem{assumption}[theorem]{Assumption}

\theoremstyle{remark}
\newtheorem{remark}[theorem]{Remark}
\newtheorem{notation}[theorem]{Notation}

\newcommand{\bH}{\mathbb{H}}
\newcommand{\bI}{\mathbb{I}}

\newcommand{\bQ}{\mathbb{Q}}
\newcommand{\bR}{\mathbb{R}}
\newcommand{\bZ}{\mathbb{Z}}

\newcommand\cA{\mathcal{A}}
\newcommand\cB{\mathcal{B}}

\newcommand\cD{\mathcal{D}}

\newcommand\cF{\mathcal{F}}
\newcommand\cH{\mathcal{H}}
\newcommand\cK{\mathcal{K}}

\newcommand\cN{\mathcal{N}}

\newcommand\cP{\mathcal{P}}
\newcommand\cQ{\mathcal{Q}}

\newcommand\cU{\mathcal{U}}

\newcommand\cM{\mathcal{M}}

\providecommand{\norm}[1]{\lVert#1\rVert}

\def\Xint#1{\mathchoice
	{\XXint\displaystyle\textstyle{#1}}%
	{\XXint\textstyle\scriptstyle{#1}}%
	{\XXint\scriptstyle\scriptscriptstyle{#1}}%
	{\XXint\scriptscriptstyle\scriptscriptstyle{#1}}%
	\!\int}
\def\XXint#1#2#3{{\setbox0=\hbox{$#1{#2#3}{\int}$}
		\vcenter{\hbox{$#2#3$}}\kern-.5\wd0}}
\def\dashint{\Xint-}%For the average integral symbol

\newcommand{\p}{\partial}
\newcommand{\epsi}{\varepsilon}

\begin{document}

\title[Mixed boundary value problem]{Optimal regularity of mixed Dirichlet-conormal  boundary value problems for parabolic operators} %with time-dependent interfacial boundary}

\author[J. Choi]{Jongkeun Choi}
\address[J. Choi]{Department of Mathematics Education, Pusan National University, Busan 46241, Republic of Korea}
	
\email{jongkeun\_choi@pusan.ac.kr}

\thanks{J. Choi was supported by the National Research Foundation of Korea (NRF) under agreement NRF-2019R1F1A1058826}

\author[H. Dong]{Hongjie Dong}	
\address[H. Dong]{Division of Applied Mathematics, Brown University, 182 George Street, Providence, RI 02912, USA}
	\email{hongjie\_dong@brown.edu}
\thanks{H. Dong was partially supported by the Simons Foundation, grant no. 709545, a Simons fellowship, grant no. 007638, and the NSF under agreement DMS-2055244.}

\author[Z. Li]{Zongyuan Li}
\address[Z. Li]{Department of Mathematics, Rutgers University, 110 Frelinghuysen Road, Piscataway, NJ 08854-8019, USA}
\email{zongyuan.li@rutgers.edu}
\thanks{Z. Li was partially supported by an AMS-Simons travel grant.}


\begin{abstract}
We obtain the regularity of solutions in Sobolev spaces for the mixed Dirichlet-conormal problem for parabolic operators in cylindrical domains with time-dependent separations, which is the first of its kind. Assuming the boundary of the domain to be Reifenberg-flat and the separation to be locally sufficiently close to a Lipschitz function of $m$ variables, where $m=0,\ldots,d-2$, with respect to the Hausdorff distance, we prove the unique solvability for $p\in (2(m+2/(m+3),2(m+2)/(m+1)))$. In the case when $m=0$, the range $p\in(4/3,4)$ is optimal in view of the known results for Laplace equations.
\end{abstract}

\maketitle

\tableofcontents
%========================================
\section{Introduction}
%========================================
In this paper, we obtain the maximal regularity for divergence form parabolic equations with mixed boundary conditions:
	\begin{equation}		\label{eqn-11291710}
		\begin{cases}
			\cP u-\lambda u=  D_ig_i+f  & \text{in }\, \cQ^T,\\
			\cB u = g_in_i& \text{on }\, \cN^T,\\
			u = 0 & \text{on }\, \cD^T,
		\end{cases}
	\end{equation}
where for some $T\in (-\infty, \infty]$, $\cQ^T=(-\infty,T)\times\Omega$ is a cylinder with the base $\Omega\subset\bR^d$ being either bounded or unbounded.
The lateral boundary of $\cQ^T$ is decomposed into two non-intersecting components $\cD^T$ and $\cN^T$ on which we impose two different types of boundary conditions.
We consider both cases when $\cD^T,\cN^T$, and their interfacial boundary (separation) $\Gamma^T$ are cylindrical and non-cylindrical.

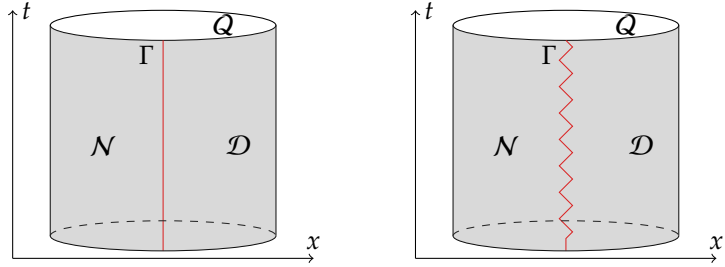
\begin{figure}[hb!]
		\begin{tikzpicture}
		\draw [->] (-2,-3.3) -- (-2, 0) node (taxis) [right] {$t$};
		\draw [->] (-2,-3.3) -- (2,-3.3) node (xaxis) [above] {$x$};
		\draw (0,-0.2) ellipse (1.5 and 0.2);%topcircle
		\draw (-1.5,-0.2) -- (-1.5,-3);%leftend
		\draw[red] (0,-0.4) -- (0,-3.2);%Gamma
		\draw (-1.5,-3) arc (180:360:1.5 and 0.2);%bottom lower half circle
		\draw [dashed] (-1.5,-3) arc (180:360:1.5 and -0.2);%bottom upper half circle
		\draw (1.5,-0.2) -- (1.5,-3);  %rightend
		\fill [gray,opacity=0.3] (-1.5,-0.2) -- (-1.5,-3) arc (180:360:1.5 and 0.2) -- (1.5,-0.2) arc (0:180:1.5 and -0.2);%fill front surface
		\node[left] at (0,-0.6) {$\Gamma$};
		\node at (1,-1.8) {$\cD$};
		\node at (-0.8,-1.8) {$\cN$};
		\node at (0.8,-0.2) {$\cQ$};
%		\node at (1.2,-0.6) {$\p_l\cQ$};
	\end{tikzpicture}		
%\hfill
\quad \quad \quad
\begin{tikzpicture}
	\draw [->] (-2,-3.3) -- (-2, 0) node (taxis) [right] {$t$};
	\draw [->] (-2,-3.3) -- (2,-3.3) node (xaxis) [above] {$x$};
	\draw (0,-0.2) ellipse (1.5 and 0.2);%topcircle
	\draw (-1.5,-0.2) -- (-1.5,-3);%leftend
	\draw[snake=zigzag, red] (0,-0.4) -- (0,-3.2);%Gamma
	\draw (-1.5,-3) arc (180:360:1.5 and 0.2);%bottom lower half circle
	\draw [dashed] (-1.5,-3) arc (180:360:1.5 and -0.2);%bottom upper half circle
	\draw (1.5,-0.2) -- (1.5,-3);  %rightend
	\fill [gray,opacity=0.3] (-1.5,-0.2) -- (-1.5,-3) arc (180:360:1.5 and 0.2) -- (1.5,-0.2) arc (0:180:1.5 and -0.2);%fill front surface
	\node[left] at (0,-0.6) {$\Gamma$};
	\node at (1,-1.8) {$\cD$};
	\node at (-0.8,-1.8) {$\cN$};
	\node at (0.8,-0.2) {$\cQ$};
%	\node at (1.2,-0.6) {$\p_l\cQ$};
\end{tikzpicture}
\caption{Domains with cylindrical and non-cylindrical $\Gamma$}
\end{figure}
We also consider the equations in a finite height cylinder, in which case, we impose the zero initial condition at $t=S$ for some finite $S<T$.
For the parabolic operator $\cP$ and the conormal derivative operator $\cB$
$$
\begin{aligned}
\cP u&=-u_t+D_i(a^{ij}D_j u+a^i u)+b^i D_i u+cu,\\
\cB u&=(a^{ij}D_j u+a^i u)n_i,
\end{aligned}
$$
where $n=(n_1,\ldots, n_d)$ is the outward unit normal to $\partial \Omega$, we always assume that the leading coefficients $(a^{ij})$ are symmetric, and that  there exists $\Lambda\in (0,1]$ satisfying
$$
a^{ij}(t,x)\xi_j\xi_i\ge \Lambda |\xi|^2, \quad |a^{ij}(t,x)|\le \Lambda^{-1}
$$
for any $\xi\in \bR^d$ and $(t,x)\in \bR\times \bR^d$.
We also always assume
%{\color{red}that  the lower-order coefficients $a^i$, $b^i$, and $c$ satisfy}
\begin{equation*}
|a^i|+|b^i|+|c|\le K
\end{equation*}
for some $K>0$.

Elliptic and parabolic equations with mixed boundary conditions arise naturally in physics and material science.
For example, when a block of floating ice is melting into liquid water, the ice-water interface maintains zero temperature (Dirichlet) while the ice-air interface is insulated (Neumann).
Such problem also has applications in the combustion theory. See, for example, \cite{MR1804512,MR1799414}.
We also refer to \cite{MR3190265} for an application in modelling exocytosis, which is a form of active transport mechanism. It is worth mentioning that the mixed problems for the heat or Laplace equations are commonly called Zaremba problem in the literature.

In contrast to the purely Dirichlet or conormal boundary value problem, solutions to mixed boundary value problems can be non-smooth near the separation $\Gamma$ even if the domain, coefficients, and boundary data are all smooth. In the literature, elliptic equations with mixed boundary conditions have been studied quite extensively both from the PDE perspective and harmonic analysis point of view. We refer the reader to \cite{MR0239272, MR0826642, MR1284808, MR1452171, MR2810833, MR3573649} and \cite{MR3034453, MR3170211, BC20,arXiv:1904.00545, arXiv:2003.10980} and the references therein. In these papers, regularity of solutions in H\"older, Sobolev, and Besov spaces as well as the non-tangential maximal function estimates were obtained.
We also refer the reader to \cite{MR3315670, MR3797618} for results about the Kato square root problem for elliptic equations and systems with mixed boundary conditions and their applications.

However, there are relatively few papers dealing with parabolic equations with mixed boundary conditions.
In \cite{MR1950984}, Skubachevskii and Shamin established the existence, uniqueness, and stability of solutions to parabolic equations in a cylindrical domain with time-independent separation $\Gamma$ under minimal regularity assumptions on initial data, by using a semigroup approach and properties of difference operators in $L_2$-based Sobolev spaces. Hieber and Rehberg \cite{MR2403322} studied systems of reaction-diffusion equations with mixed Dirichlet-Neumann boundary conditions in 2D and 3D cylindrical Lipschitz domains satisfying Gr\"oger's regular condition. Assuming an implicit topological isomorphism condition on the second-order operator, they proved the maximal $L_p$ estimates by using harmonic analysis and a heat-kernel method. See also \cite{MR2541414} for a further result about equations with nonhomogeneous boundary data as well as an earlier result in \cite{MR0990595} about the elliptic mixed boundary value problem.
We also mention the work \cite{MR1444765,MR3804727} for parabolic equations in non-cylindrical domains. In \cite{MR1444765} Savar\'{e} considered parabolic equations in a non-cylindrical domain with $C^{1,1}$ boundary and separation $\Gamma$. Under certain condition on the excess of $\Gamma$ with respect to $t$, he introduced an approximation approach to general abstract evolution equations in $L_2$-based Sobolev spaces and  obtained optimal regularity under quite weak assumptions on the data. %, u_t,Au\in L_2, $Du\in L_\infty(L_2)\cap B_{2,\infty}^{1/2}(L_2)
Recently in \cite{MR3804727}, Kim and Cao studied linear and semilinear parabolic equations in non-cylindrical domains with the mixed Dirichlet, Neumann, and Robin conditions and Lipschitz leading coefficients. They considered smooth domains which are $C^1$ in $t$ and $C^2$ in $x$ and general $\cD^T$ and $\cN^T$, and proved the solvability in $L_2$-based Sobolev spaces.
In \cite{MR3582221}, a maximal regularity result was established for parabolic equations with time irregular coefficients and mixed boundary conditions when the exponents are in a neighborhood of $2$. See also \cite{MR3659370} for an optimal $L^p$ maximal regularity result when $p\le 2$ and the coefficients are BV in the $x$ variable. We remark that in all these work, either $p$ is assumed to be $2$ or an implicit condition is imposed on the operator, so that $p$ needs to be sufficiently close to $2$.
%For other previous results about mixed boundary value problem, we refer the reader to \cite{MR3126756, MR3277902}  for semilinear parabolic equations with {\color{blue}a} memory term and \cite{MR3348258} for fractional parabolic equations.

In this paper, we consider parabolic equations in a cylindrical domain with rough boundary and separation and with vanishing mean oscillation (VMO) coefficients. Let $m=0,\ldots,d-2$ be an integer. Our main result, Theorem \ref{MT1}, reads that if $\partial\Omega$ is Reifenberg-flat and the separation $\Gamma$ is locally sufficiently close to a Lipschitz function of $m$ variables with respect to the Hausdorff distance, then for any $p\in \big(\frac{2(m+2)}{m+3}, \frac{2(m+2)}{m+1}\big)$, $f,g_i\in L_p(\cQ^T)$, and sufficiently large $\lambda$, there is a unique solution $u\in \cH^1_p(\cQ^T)$ to \eqref{eqn-11291710}. See Section \ref{sec2.1} for the definitions of function spaces. In the special case when $\Gamma$ is Reifenberg-flat of co-dimension $2$, i.e., $m=0$, we get the solvability when $p\in (4/3,4)$, which is an optimal range in view of the known results for the Laplace equation in \cite{MR0239272} and for elliptic equations in \cite{arXiv:1904.00545}.
%Furthermore, under the same conditions, in Theorem \ref{thm-201022-0843} we also obtain the solvability in mixed-norm space $L_q^t(L_p^x)$ for any $q\in [p,\infty)$ when $p\in \big(\frac{2(m+2)}{m+3}, 2\big]$. This shows that in the case of mixed boundary conditions,
%the higher-in-time integrability of right-hand side still implies the same higher integrability of the solution. To the best of our knowledge, these are the first results about the maximal regularity for parabolic equations with mixed boundary conditions even in the case when $\partial\Omega$ and $\Gamma$ are smooth and $\Gamma$ is time-independent.

Let us give a brief description of the proof. In the first step, we show the $L_2$ solvability by approximating $\Gamma$ by piecewise time-independent separations and $\Omega$ by extension domains, and solve the approximating equations using the Galerkin method. Note that one cannot directly test the equation with $u$ because the usual Steklov average argument is not applicable here.
The second step is to derive an estimate of the form
$$
\|Du\|_{L^{x'}_pL_\infty^{t,x''}(Q_{1/2}^+)}\le  C\|Du\|_{L_1(Q_{1}^+)},\quad
\forall p\in \bigg[2,\frac{2(m+2)}{m+1}\bigg),
$$
where $u$ is a weak solution to a homogeneous equation with constant coefficients in the half cylinder $Q_1^+=\{(t,x):t\in (-1,0),|x|<1,x^1>0\}$. It satisfies the mixed homogeneous boundary conditions on the flat boundary $\{x^1=0\}$ with a time-independent separation $\Gamma$ defined by a $\{x^2=\phi(x^3,\ldots,x^{m+2})\}$, where $\phi$ is a Lipschitz function, $x'=(x^1,\ldots,x^{m+2})$ and $x''=(x^{m+3},\ldots,x^d)$. This will be achieved by a boundary Caccioppoli type inequality and the corresponding elliptic estimate obtained in \cite{arXiv:2003.10980}, which in turn is a consequence of the Besov type estimate established in \cite{MR1452171,BC20}. Since $\partial\Omega$ and $\Gamma$ are not smooth, the usual flattening boundary argument does not work in our case. To approximate the domain by domains with flat boundaries and separations, the third step in the proof is to apply the cutoff and reflection argument first used in \cite{MR2835999,MR3013054} for equations with the pure Dirichlet or conormal boundary condition. Such argument was also used recently in \cite{MR3809039,arXiv:2003.10980} for elliptic equations with mixed boundary conditions. Here an additional difficulty is the extra $u_t$ term which appears on the right-hand side after the reflection. For this, we use a delicate multi-step decomposition procedure. Finally, we complete the proof of Theorem \ref{MT1} by utilizing a level set argument introduced by Caffarelli and Peral \cite{MR1486629} and the ``crawling of ink spots" lemma due to Krylov and Safonov \cite{MR0579490,MR0563790}. %To prove the mixed-norm estimate in Theorem \ref{thm-201022-0843}, we first localize the estimate in Theorem \ref{MT1} to obtain a reverse H\"older's inequality. We then borrow an idea of Krylov \cite{MR2352490} to apply the level set argument in the $t$-variable only.

The rest of the paper is organized as follows. In Section \ref{sec2}, we introduce the notation and function spaces, and then give our main result, Theorem \ref{MT1} for the solvability of \eqref{eqn-11291710} in Sobolev spaces. %and Theorem \ref{thm-201022-0843} for the solvability in mixed-norm Sobolev spaces.
In Section \ref{sec3}, we prove several auxiliary results including the $L_2$ solvability, Poincar\'e and embedding inequalities suitable to our problem, and a reverse H\"older's inequality for weak solutions to the mixed boundary value problem. Section \ref{sec4} is devoted to the boundary estimates for equations with constant coefficients near a curved boundary or a flat boundary. In Section \ref{S5}, we give the proof of Theorem \ref{MT1}. %and \ref{thm-201022-0843}, respectively.

%========================================
\section{Notation and main results}     \label{sec2}
%========================================

%========================================
\subsection{Notation}       \label{sec2.1}
%========================================

Throughout the paper, we always assume that $\Omega$ is a domain (open and connected, but not necessarily bounded) in $\bR^d$ and $\cQ=(-\infty,\infty)\times \Omega$ is an infinitely long cylinder, which is a subset of
$$
\bR^{d+1}=\{X=(t,x): t\in \bR, \, x=(x^1,\ldots, x^d)\in \bR^d\}.
$$
We also assume that the  boundary of $\cQ$, denoted by  $\partial \cQ=(-\infty, \infty)\times \partial \Omega$, is divided into two disjoint portions $\cD$ and $\cN$, separated by $\Gamma$. More precisely, let $\cD\subset \partial\cQ$ be an open set (relative to $\partial\cQ$) and
$$
\cN = \partial\cQ\setminus\cD,\quad \Gamma=\overline{\cD}\cap\overline{\cN}.
$$
Note that the separation $\Gamma$ between $\cD$ and $\cN$ is time-dependent unless explicitly specified otherwise.
In Section \ref{S_2_2}, we will impose certain regularity assumptions on $\partial \Omega$ and $\Gamma$.

For $T\in (-\infty, \infty]$, we define
$$
\cQ^T=\{X\in \cQ: t< T\}
$$
and similarly define $\cD^T$, $\cN^T$, and $\Gamma^T$.
For $R>0$, we set
$$
B_R(x)=\{y\in \bR^d:|x-y|<R\}, \quad \Omega_R(x)=\Omega\cap B_R(x),
$$
$$
Q_R(X)=(t-R^2,t)\times B_R(x), \quad \cQ_R(X)=\cQ\cap Q_R(X),
$$
$$
\bQ_R(X)=(t-R^2,t+R^2)\times B_R(x).
$$
We use the abbreviations $B_R=B_R(0)$ and $\Omega_R=\Omega_R(0)$, etc.
For a function $f$ on $Q$, we denote
$$
(f)_Q=\frac{1}{|Q|}\int_{Q} f \,dX=\dashint_{Q} f\,dX.
$$

Now we introduce function spaces and the notion of weak solutions to mixed boundary value problems. To start with, we use the unified notation for cylinders with finite or infinite height: for given $S,T$ with $-\infty\le S<T\le \infty$, we set
\begin{equation}		\label{200508@eq4}
\begin{aligned}
\tilde{\cQ}&=\{X\in \cQ: S<t<T\}, \\
\tilde{\cD}&=\{X\in \cD: S<t<T\}, \\
\tilde{\cN}&=\{X\in \cN: S<t<T\}.
\end{aligned}
\end{equation}
We write $X=(t,x)=(t,x',x'')\in \bR^{d+1}$, where $x'\in\bR^{d_1}$ and $x''\in \bR^{d_2}$, $d_1+d_2=d$.
For $p,q\in [1, \infty)$, we define $L_{q}^{(t,x'')}L_p^{x'}(\tilde{\cQ})$ to be the set of all functions $u$ such that
$$
\|u\|_{L_q^{(t,x'')}L_p^{x'}(\tilde{\cQ})}:=\Bigg(\int_{\bR^{d_2+1}}\Bigg(\int_{\bR^{d_1}} |u|^p \bI_{\tilde{\cQ}}\,dx'\Bigg)^{q/p}dx''\, dt\Bigg)^{1/q}<\infty,
$$
where $\bI_{\tilde{\cQ}}$ is the usual characteristic function.
Similarly, we define $L_q^{(t,x'')}L_p^{x'}(\tilde{\cQ})$ with $p=\infty$ or $q=\infty$, and
$L_p^{x'}L_q^{(t,x'')}(\tilde{\cQ})$.
We abbreviate
$$
L^t_qL^x_p(\tilde{\cQ})=L_{q,p}(\tilde{\cQ}), \quad L_{p,p}(\tilde{\cQ})=L_p(\tilde{\cQ}).
$$
Let $C^\infty_{\tilde{\cD}}(\tilde{\cQ})$ be the set of all infinitely differentiable functions on $\bR^{d+1}$ having a compact support in $[S,T]\times \overline{\Omega}$ and vanishing in a neighborhood of $\tilde{\cD}$.
We denote by $W^{0,1}_{q,p, \tilde{\cD}}(\tilde{\cQ})$ and $W^{1,1}_{q,p, \tilde{\cD}}(\tilde{\cQ})$ the closures of $C^\infty_{\tilde{\cD}}(\tilde{\cQ})$ in $W^{0,1}_{q,p}(\tilde{\cQ})$ and $W^{1,1}_{q,p}(\tilde{\cQ})$, respectively, where
$$
W^{0,1}_{q,p}(\tilde{\cQ})=\{u: u,  Du\in L_{q,p}(\tilde{\cQ})\}, \quad W^{1,1}_{q,p}(\tilde{\cQ})=\{u: u,  Du, u_t\in L_{q,p}(\tilde{\cQ})\}.
$$
We write $W^{0,1}_{p,p}(\tilde{\cQ})=W^{0,1}_p(\tilde{\cQ})$, etc.

By $u\in \bH^{-1}_{p, \tilde{\cD}}(\tilde{\cQ})$ we mean that there exist $g=(g_1,\ldots, g_d)\in L_{p}(\tilde{\cQ})^d$ and $f\in L_{p}(\tilde{\cQ})$ such that
$$
u=D_i g_i+f \quad \text{in }\, \tilde{\cQ}, \quad g_in_i=0 \quad \text{on }\, \tilde{\cN},
$$
where $n=(n_1,\ldots,n_d)$ is the outward unit normal to $\partial \Omega$, in the distribution sense
and the norm
$$
\|u\|_{\bH^{-1}_{p, \tilde{\cD}}(\tilde{\cQ})}=\inf\big\{ \|g\|_{L_{p}(\tilde{\cQ})}+\|f\|_{L_{p}(\tilde{\cQ})} : u=D_ig_i+f\, \text{ in }\, \tilde{\cQ}, \, g_in_i=0 \, \text{ on }\, \tilde{\cN}\big\}
$$
is finite.
We set
$$
\cH_{p, \tilde{\cD}}^1(\tilde{\cQ})=\big\{u: u\in W^{0,1}_{p, \tilde{\cD}}(\tilde{\cQ}), \, u_t\in \bH^{-1}_{p, \tilde{\cD}}(\tilde{\cQ})\big\}
$$
equipped with a norm
$$
\|u\|_{\cH^1_{p, \tilde{\cD}}(\tilde{\cQ})}
=\|u\|_{W^{0,1}_{p}(\tilde{\cQ})}+\|u_t\|_{\bH^{-1}_{p, \tilde{\cD}}(\tilde{\cQ})}.
$$

When $\tilde{\cQ}$ has a finite height (i.e., $-\infty<S<T<\infty$), we define $V_{2}(\tilde{\cQ})$ as the set of all functions in $W^{0,1}_{2}(\tilde{\cQ})$ having the following finite norm
\begin{equation*}
\norm{u}_{V_2(\tilde{\cQ})}=\operatorname*{ess\,sup}_{S< t < T}\,\norm{u(t,\cdot)}_{L_2(\Omega)}+\norm{Du}_{L_2(\tilde{\cQ})}.
\end{equation*}
When $\tilde{\cQ}$ has an infinite height, we define $V_{2}(\tilde{\cQ})$  as the set of $u\in V_2(\tilde{\cQ}\cap \{(t,x):|t|<T_0\})$  for all $T_0>0$ having the finite norm $\|u\|_{V_2(\tilde{\cQ})}$.
We then denote by $V_{2, \tilde{\cD}}(\tilde{\cQ})$ the closure of $W^{0,1}_{2,\tilde{\cD}}(\tilde{\cQ})$ in $V_2(\tilde{\cQ})$.

We write $u\in L_{q, p, \rm{loc}}(\tilde{\cQ})$ if $\eta u\in L_{q,p}(\tilde{\cQ})$ for any infinitely differentiable function $\eta$ on $\bR^{d+1}$ having a compact support, and
similarly define $\cH^1_{p, \tilde{\cD},{\rm{loc}}}(\tilde{\cQ})$, etc.
\begin{definition}
We say that $u\in \cH_{p,\tilde{\cD},{\rm{loc}}}^1(\tilde{\cQ})$ satisfies the mixed boundary value problem
\begin{equation}		\label{200508@eq2}
\begin{cases}
\cP u = D_ig_i + f & \text{in }\, \tilde{\cQ},\\
\cB u = g_in_i& \text{on }\, \tilde{\cN},\\
u = 0 & \text{on }\, \tilde{\cD},\\
\end{cases}
\end{equation}
where $g=(g_1,\ldots, g_d)\in L_{p,{\rm{loc}}}(\tilde{\cQ})^d$ and $f\in L_{p, {\rm{loc}}}(\tilde{\cQ})$ if
$$
\int_{\tilde{\cQ}} u\varphi_t\,dX+\int_{\tilde{\cQ}} (-a^{ij}D_j uD_i \varphi-a^i uD_i \varphi+b^iD_i u\varphi + c u\varphi)\,dX = \int_{\tilde{\cQ}} (-g_i D_i \varphi+f\varphi)\,dX
$$
for all $\varphi\in C^\infty_{\tilde{\cD}}(\tilde{\cQ})$ that vanishes for $t=S$ and $T$.
For $S>-\infty$, we also say that $u\in \cH^1_{p,\tilde{\cD}, {\rm{loc}}}(\tilde{\cQ})$ satisfies \eqref{200508@eq2} with the initial condition $u(S,\cdot)=\psi$ on $\Omega$,
where $\psi\in L_{p, {\rm{loc}}}(\Omega)$ if
\begin{equation}		\label{200508@eq3}
\begin{gathered}
\int_{\tilde{\cQ}} u\varphi_t\,dX+\int_{\tilde{\cQ}} (-a^{ij}D_j uD_i \varphi-a^i uD_i \varphi+b^iD_i u\varphi + c u\varphi)\,dX \\
= -\int_{\tilde{\cQ}} (g_i D_i \varphi-f\varphi)\,dX-\int_\Omega \psi \varphi(S,\cdot)\,dx
\end{gathered}
\end{equation}
for all $\varphi\in C^\infty_{\tilde{\cD}}(\tilde{\cQ})$ that vanishes for $t=T$.
\end{definition}

%========================================
\subsection{Main result}		\label{S_2_2}
%========================================

We impose the following regularity assumptions on the domain and the leading coefficients. The first one is the so-called Reifenberg flat conditions on $\p\Omega$ and $\Gamma$.

\begin{assumption}[$\gamma; m, M$]		\label{A11}
Let $m\in \{0,1,\ldots, d-2\}$ and $M\in (0, \infty)$.
\begin{enumerate}[$(a)$]
\item
For any $x_0\in \partial \Omega$ and $R\in (0, R_0]$, there is a coordinate system depending on $x_0$ and $R$ such that in this coordinate system, we have
\begin{equation}		\label{200429@eq1}
\{y: y^1>x_0^1+\gamma R\}\cap B_R(x_0)\subset \Omega_R(x_0)\subset \{y: y^1>x_0^1-\gamma R\}\cap B_R(x_0).
\end{equation}

\item
For any $X_0=(t_0,x_0)\in \Gamma$ and $R\in (0, R_0]$,   there exist
a spatial coordinate system and a Lipschitz function $\phi$ of $m$ variables with Lipschitz constant $M$,  such that in the new coordinate system (called the coordinate system associated with $(X_0, R)$), we have  \eqref{200429@eq1},
$$
\big(\partial \cQ\cap \bQ_R(X_0) \cap \{(s,y): y^2>\phi(y^3,\ldots,y^{m+2})+\gamma R\} \big)\subset \cD,
$$
$$
\big(\partial \cQ \cap \bQ_R(X_0)\cap  \{(s,y): y^2<\phi(y^3,\ldots, y^{m+2})-\gamma R\}\big)\subset \cN,
$$
and
$$
\phi(x_0^3,\ldots, x_0^{m+2})=x_0^2.
$$
Here, if $m=0$, then the function $\phi$ is  understood as the constant function $\phi\equiv x_0^2$.
\end{enumerate}
\end{assumption}

The second is the small BMO condition on the leading coefficients.

\begin{assumption}[$\theta$]\label{A2}
For any $X_0\in \overline{\cQ}$ and $R\in (0, R_0]$, we have
$$
\dashint_{\cQ_R(X_0)}|a^{ij}(X)-(a^{ij})_{\cQ_R(X_0)}|\,dX\le \theta.
$$
\end{assumption}

Our theorems deal with cylinders with infinite or finite height. We rewrite \eqref{eqn-11291710} in the unified form

\begin{equation}		\label{200812@eq3}
	\begin{cases}
		\cP u-\lambda u=  D_ig_i+f  & \text{in }\, \tilde{\cQ},\\
		\cB u = g_in_i& \text{on }\, \tilde{\cN},\\
		u = 0 & \text{on }\, \tilde{\cD},
	\end{cases}
\end{equation}
where $\tilde{\cQ}, \tilde{\cN}$, and $\tilde{\cD}$ are defined  in \eqref{200508@eq4}, and $-\infty\le S<T\le \infty$.
The main theorem of the paper is the following solvability result for \eqref{200812@eq3} in Sobolev spaces.

\begin{theorem}		\label{MT1}
Let
$R_0\in (0,1]$,  $m\in \{0,1,\ldots,d-2\}$, $M\in (0, \infty)$, and  $p\in \big(\frac{2(m+2)}{m+3}, \frac{2(m+2)}{m+1}\big)$.
There exist constants $\gamma,\theta\in (0,1)$ and $\lambda_0\in (0, \infty)$ with
$$
(\gamma, \theta)=(\gamma, \theta)(d, \Lambda, M, p), \quad \lambda_0=\lambda_0(d,\Lambda, M, p, K, R_0),
$$
such that if Assumptions \ref{A11} $(\gamma;m,M)$ and \ref{A2} $(\theta)$ are satisfied with  these $\gamma$ and $\theta$, then the following assertions hold.
\begin{enumerate}[$(a)$]
\item
When $S=-\infty$, for any $\lambda\ge \lambda_0$, $g=(g_1,\ldots, g_d)\in L_{p}(\tilde{\cQ})^d$, and $f\in L_{p}(\tilde{\cQ})$, there exists a unique solution $u\in \cH^1_{p, \tilde{\cD}}(\tilde{\cQ})$ to \eqref{200812@eq3}, which satisfies
\begin{equation}		\label{200812@eq2}
\|Du\|_{L_p(\tilde{\cQ})}+\lambda^{1/2} \|u\|_{L_p(\tilde{\cQ})}\le C\|g\|_{L_p(\tilde{\cQ})}+C\lambda^{-1/2}\|f\|_{L_p(\tilde{\cQ})},
\end{equation}
where $C=C(d, \Lambda, M, p)$.
\item
For the initial boundary value problem on a cylindrical domain of finite height, we can take $\lambda=0$, i.e., when $S=0$ and $T\in (0, \infty)$, for any $g=(g_1,\ldots, g_d)\in L_{p}(\tilde{\cQ})^d$ and $f\in L_{p}(\tilde{\cQ})$, there exists a unique solution $u\in \cH^1_{p, \tilde{\cD}}(\tilde{\cQ})$ to \eqref{200812@eq3} with $\lambda=0$ and the initial condition $u(0,\cdot)\equiv 0$ on $\Omega$.
Moreover, we have
$$
\|u\|_{\cH^{1}_{p, \tilde{\cD}}(\tilde{\cQ})}\le C\|g\|_{L_p(\tilde{\cQ})}+C\|f\|_{L_p(\tilde{\cQ})},
$$
where $C=C(d,\Lambda, M, p, K, R_0, T)$.
\end{enumerate}
\end{theorem}
\section{Auxiliary results} \label{sec3}
%========================================

Throughout this paper, we use the following notation.

\begin{notation}
For nonnegative (variable) quantities $A$ and $B$, we denote $A \lesssim B$ if there exists a generic positive constant $C$ such that $A \le CB$.
We add subscript letters like $A\lesssim_{a,b} B$ to indicate the dependence of the implicit constant $C$ on the parameters $a$ and $b$.
\end{notation}

%========================================
\subsection{\texorpdfstring{$L_2$}{L2} estimate}
%========================================

In this subsection, we prove the solvability of the mixed boundary value problem in $V_{2, \tilde{\cD}}(\tilde{\cQ})$
under the following assumption that the boundary portions $\cD$ and $\cN$ vary continuously in $t$, which is weaker than the condition $(b)$ in Assumption \ref{A11} $(\gamma;m,M)$; see Lemma \ref{211112@lem1}.

\begin{assumption}\label{ass-0301-2356}
For any $\epsi>0$ and $L>0$, there exist a time partition
$$
\max\{S,-L\}=t_0<t_1<\cdots<t_{n}=\min\{T,L\}
$$
and decompositions $\p\Omega=D^{t_k}\cup N^{t_k}$, $k\in \{1,\ldots, n\}$, such that
\begin{equation}		\label{200618@eq3}
\cD(t)\subset D^{t_k},\quad H^d(\cD(t), D^{t_k})<\epsi,\quad \forall t\in[t_{k-1},t_k),
\end{equation}
where $\cD(t)=\{x\in \partial \Omega: (t,x)\in \cD\}$ and $H^d$ is the usual $d$-dimensional Hausdorff distance.
\end{assumption}
	Note that Assumption \ref{ass-0301-2356} excludes the following possibility:
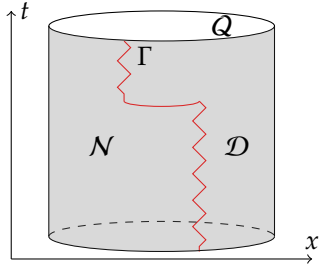
\begin{figure}[!hb]
\begin{center}
\begin{tikzpicture}
\draw [->] (-2,-3.3) -- (-2, 0) node (taxis) [right] {$t$};
\draw [->] (-2,-3.3) -- (2,-3.3) node (xaxis) [above] {$x$};
\draw (0,-0.2) ellipse (1.5 and 0.2);%topcircle
\draw (-1.5,-0.2) -- (-1.5,-3);%leftend
\draw[snake=zigzag, red] (-0.5,-0.4) -- (-0.5,-1.2);%Gamma
\draw[red] (-0.5,-1.2) arc (180:360:0.5 and 0.07);%Gamma
\draw[snake=zigzag, red] (0.5,-1.2) -- (0.5,-3.2);%Gamma
\draw (-1.5,-3) arc (180:360:1.5 and 0.2);%bottom lower half circle
\draw [dashed] (-1.5,-3) arc (180:360:1.5 and -0.2);%bottom upper half circle
\draw (1.5,-0.2) -- (1.5,-3);  %rightend
\fill [gray,opacity=0.3] (-1.5,-0.2) -- (-1.5,-3) arc (180:360:1.5 and 0.2) -- (1.5,-0.2) arc (0:180:1.5 and -0.2);%fill front surface
\node[left] at (0,-0.6) {$\Gamma$};
\node at (1,-1.8) {$\cD$};
\node at (-0.8,-1.8) {$\cN$};
\node at (0.8,-0.2) {$\cQ$};
% \node at (1.2,-0.6) {$\p_l\cQ$};
\end{tikzpicture}
\end{center}
\end{figure}

In the proposition below, we do not impose any regularity  assumptions on $a^{ij}=a^{ij}(t,x)$. It generalizes the classical solvability result in \cite{MR0241822} by allowing a rough $\Omega$ and a time-varying $\Gamma$.

\begin{proposition} 		\label{200810@prop1}
Let $\Omega$ be a bounded or unbounded domain in $\bR^d$.
Under Assumption \ref{ass-0301-2356}, there exists $\lambda_0\ge 0$ depending only on $d$, $\Lambda$, and $K$ such that the following assertions hold.
For any  $\lambda\ge \lambda_0$, $g=(g_1,\ldots, g_d)\in L_2(\tilde{\cQ})^d$, $f\in L_2(\tilde{\cQ})$, and $\psi\in L_2(\Omega)$,
there exists a unique $u\in V_{2,\tilde{\cD}}(\tilde{\cQ})$ satisfying
\begin{equation}		\label{200324@eq1}
\begin{cases}
\cP u-\lambda u = D_ig_i +f & \text{in }\, \tilde{\cQ},\\
\cB u = g_in_i& \text{on }\, \tilde{\cN},\\
u = 0 & \text{on }\, \tilde{\cD},\\
u=\psi  &\text{on }\, \{S\}\times \Omega \quad \text{if }\, S>-\infty.
\end{cases}
\end{equation}
Moreover, we have
\begin{equation} \label{est-0302-0000}
\norm{Du}_{L_2(\tilde{\cQ})}  + \lambda^{1/2}\norm{u}_{L_2(\tilde{\cQ})} \lesssim_{d, \Lambda} \norm{g}_{L_2(\tilde{\cQ})} + \lambda^{-1/2}\norm{f}_{L_2(\tilde{\cQ})}+\|\psi\|_{L_2(\Omega)}
\end{equation}
and
\begin{equation}
                        \label{eq8.02}
\norm{u}_{L_{\infty,2}(\tilde{\cQ})} \le \|\psi\|_{L_2(\Omega)}+C(d,\Lambda)\big(\norm{g}_{L_2(\tilde{\cQ})} + \lambda^{-1/2}\norm{f}_{L_2(\tilde{\cQ})}\big).
\end{equation}
Here the initial condition $u=\psi$ and the term $\|\psi\|_{L_2(\Omega)}$ in the estimates only appear when $S>-\infty$.
\end{proposition}

The rest of this subsection is devoted to the proof of Proposition \ref{200810@prop1}. In the following, we call a domain $\Omega(\subset \bR^d)$ an extension domain if it is bounded and admits an bounded extension $W^1_2(\Omega)\rightarrow W^1_2(\mathbb{R}^d)$. In particular, the compact embedding $W^1_{2,D}(\Omega)\hookrightarrow L_2(\Omega)$ holds. The following lemma should be classical, which we state explicitly for  completeness.
\begin{lemma}		\label{190619@lem1}
Let $\Omega$ be an extension domain in $\bR^d$ and $D\subset \partial \Omega$. Then there exists an orthogonal basis $\{w_i\}$ for $W^{1}_{2,D}(\Omega)$ satisfying
$$
\int_\Omega w_i  w_j\,dx=\delta_{ij},
$$
where $\delta_{ij}$ is the Kronecker delta symbol.
\end{lemma}

\begin{proof}
The lemma follows from the same argument used in \cite[\S 6.5]{MR2597943}, by finding Laplacian eigenfunctions with the help of the compact embedding $W^1_2(\Omega)\hookrightarrow L_2(\Omega)$ and the spectral theory for compact operators.
We omit the details.
\end{proof}

Using Lemma \ref{190619@lem1} and the Galerkin method, we obtain the following $L_2$ solvability of the mixed problem \eqref{200324@eq1} with a cylindrical separation.

\begin{lemma}		\label{200810@LEM1}
Proposition \ref{200810@prop1} holds when $\tilde{\cD}$ and $\tilde{\cN}$ are time-independent.
\end{lemma}

\begin{proof}
Approximating $g$ and $f$ by $L_2$ functions with compact support in time, it suffices to consider the case when $-\infty<S<T<\infty$. When the spatial domain $\Omega$ is a bounded and regular (i.e., an extension domain), the lemma follows from the standard Galerkin method  together with Lemma \ref{190619@lem1}. The energy inequalities \eqref{est-0302-0000} and \eqref{eq8.02} are standard, with constants independent of the spatial domain. See, for example, \cite[\S 4, Chapter III]{MR0241822}.

When $\Omega$ is unbounded or irregular, we take a sequence of expanding smooth domains $\Omega^k\subset \Omega\cap B_k$ such that for any $x\in (\partial \Omega^k)\cap B_k$, $\operatorname{dist}(x,\partial \Omega)<1/k$, i.e., $\Omega^k\nearrow \Omega$ as $k\to \infty$.
Set
$$
N^k=\{x\in \partial \Omega^k: \operatorname{dist}(x, N)<1/k\},
$$
where $N=\{x\in \partial \Omega:(t,x)\in \tilde{\cN}\}$.
We also set
$$
\mathsf{Q}^k= (S,T)\times \Omega^k,\quad \mathsf{N}^k = (S,T)\times N^k,\quad \mathsf{D}^k=(S,T)\times (\p\Omega^k\setminus N^k).
$$
We may assume that $k$ is sufficiently large so that $\mathsf{D}^k$ and $\mathsf{N}^k$ are not empty.
Now we solve the mixed problem in $\mathsf{Q}^k$ with zero Dirichlet boundary condition on $\mathsf{D}^k$, homogeneous conormal boundary condition on $\mathsf{N}^k$, and initial data $\psi$. Since $\Omega^k$ is an extension domain, the solution $u^k\in V_{2, \mathsf{D}^k}(\mathsf{Q}^k)$ exists and is uniformly bounded.
We extend $u^k$ to be zero, still denoted by $u^k$, on
$$
\tilde{\mathsf{Q}}^k:=(S,T)\times\{ x\in \Omega\setminus \Omega^k:\, \operatorname{dist}(x,N)>2/k\}.
$$
Then $u^k$ is in $V_2(\mathsf{Q}^k\cup \tilde{\mathsf{Q}}^k)$ and vanishes on
$$
\{(t,x)\in \tilde\cD: \operatorname{dist}(x,N)>2/k\}.
$$
By a diagonal argument, we can pick a  subsequence, again denoted by $u^k$, such that as $k\to \infty$,
$$
u^{k} \rightharpoonup u, \quad Du^{k} \rightharpoonup Du \quad \text{weakly in }\, L_2((S,T)\times K),
$$
$$
u^{k}\overset{\ast}{\rightharpoonup} u \quad \text{weakly$^*$ in }\, L_{\infty,2}((S,T)\times K)
$$
for any compact set $K\subset \Omega$.
Due to the estimates \eqref{est-0302-0000} and \eqref{eq8.02}, which are uniform with respect to $k$, by taking the limit in the weak formulation and using H\"older's inequality, we see that $u\in V_{2,\tilde\cD}(\tilde{Q})$ satisfies \eqref{200324@eq1} as well as the estimates \eqref{est-0302-0000} and \eqref{eq8.02}.
The lemma is proved.
\end{proof}

We are ready to prove Proposition \ref{200810@prop1}.

\begin{proof}[Proof of Proposition \ref{200810@prop1}]
Again, by the standard approximation argument, it suffices to consider the case when $-\infty<S<T<\infty$.
Let $\varepsilon>0$ be given.
By Assumption \ref{ass-0301-2356}, there exists
a time partition
$$
S=t_0<t_1<\cdots<t_n=T
$$
and decompositions $\partial \Omega=D^{t_k}\cup N^{t_k}$, $k\in \{1,2,\ldots,n\}$, such that \eqref{200618@eq3} holds.
From Lemma \ref{200810@LEM1}, for $k\in \{1,2,\ldots, n\}$,  there exists a unique
$u^k_{\varepsilon}\in V_2^{0,1}((t_{k-1},t_k)\times \Omega)$
satisfying
\begin{equation}		\label{200307@A1}
\begin{cases}
\cP u_\varepsilon^k-\lambda u_\varepsilon^k = D_ig_i +f & \text{in }\, (t_{k-1},t_k)\times \Omega,\\
\cB u_{\varepsilon}^k = g_in_i& \text{on }\, (t_{k-1},t_k)\times  N^{t_k},\\
u_\varepsilon^k = 0 & \text{on }\, (t_{k-1},t_k)\times D^{t_k},\\
u_\varepsilon^k=u^{k-1}_{\varepsilon} & \text{on }\, \{t_{k-1}\}\times \Omega,
\end{cases}
\end{equation}
where $u^0_{\varepsilon}=\psi$, with the estimates
$$
\begin{aligned}
&\|Du_\varepsilon^k\|_{L_2((t_{k-1},t_k)\times \Omega)}
+\lambda^{1/2} \|u_\varepsilon^k\|_{L_2((t_{k-1},t_k)\times \Omega)}\\
&\lesssim_{d,\Lambda} \|g\|_{L_2((t_{k-1},t_k)\times \Omega)}+\lambda^{-1/2}\|f\|_{L_2((t_{k-1},t_k)\times \Omega)}+\|u_\varepsilon^{k-1}(t_{k-1},\cdot)\|_{L_2(\Omega)}
\end{aligned}
$$
and
\begin{align*}
 &\|u_\varepsilon^k(t,\cdot)\|_{L_{\infty,2}((t_{k-1}, t_k)\times \Omega)}\\
 &\le \|u_\varepsilon^{k-1}(t_{k-1},\cdot)\|_{L_2(\Omega)}+
 C(d,\Lambda)\big(\|g\|_{L_2((t_{k-1},t_k)\times \Omega)}+\lambda^{-1/2}\|f\|_{L_2((t_{k-1},t_k)\times \Omega)}\big).
\end{align*}
We then see that  the function $u_\varepsilon$ defined by
$$
u_\varepsilon=\sum_{k=1}^{n} u^k_{\varepsilon}\bI_{(t_{k-1},t_k)\times \Omega} \quad \text{in }\, \tilde{\cQ}=(S,T)\times \Omega
$$
belongs to $V_{2, \tilde{\cD}}(\tilde{\cQ})$ and that the estimates \eqref{est-0302-0000} and \eqref{eq8.02} hold with $u_{\varepsilon}$ in place of $u$.
Thus from the weak compactness and Alaoglu's theorems, there exist a subsequence of $\{u_\varepsilon\}$, denoted by $\{u_{\varepsilon_i}\}$, and a function $u\in V_{2,\tilde{\cD}}(\tilde{\cQ})$ such that
$$
u_{\varepsilon_i} \rightharpoonup u, \quad Du_{\varepsilon_i} \rightharpoonup Du \quad \text{weakly in }\, L_2(\tilde{\cQ}),
$$
$$
u_{\varepsilon_{i}}\overset{\ast}{\rightharpoonup} u \quad \text{weakly$^*$ in }\, L_{\infty,2}(\tilde{\cQ}).
$$
Moreover, $u$ satisfies  \eqref{est-0302-0000} and \eqref{eq8.02}.

To prove that  the limit function $u$ satisfies  \eqref{200324@eq1}, we let $\varphi\in C^\infty_{\tilde{\cD}}(\tilde{\cQ})$ vanishing for $t=T$.
Since $\varphi$ vanishes in a neighborhood of $\tilde{\cD}$ and has the compact support, we see that
$$
\operatorname{dist}(\operatorname{supp}\varphi, \tilde{\cD})>0.
$$
Hence, for sufficiently large $i$, $\operatorname{dist}(\operatorname{supp}\varphi, \tilde{\cD})>\epsi_i$, which in turn implies that $\varphi$ vanishes in a neighborhood of
$$
\bigcup_{k=1}^{n_{i}} (t_{k-1}^{(i)},t_k^{(i)})\times  D^{t_k^{(i)}},
$$
where $t_k^{(i)}$ and $D^{t_k^{(i)}}$ are given in Assumption \ref{ass-0301-2356}.
By testing  \eqref{200307@A1} with $\varphi$, we have
$$
-\int_\Omega u_{\varepsilon_{i}}^k(t_k^{(i)},\cdot)\varphi(t_k^{(i)},\cdot)\,dx+\int_{t_{k-1}^{(i)}}^{t_k^{(i)}}\int_{\Omega} u_{\varepsilon_{i}}^k \varphi_t\,dx\,dt
$$
$$
+\int_{t_{k-1}^{(i)}}^{t_k^{(i)}}\int_{\Omega} (-a^{ij}D_j u_{\varepsilon_{i}}^kD_i \varphi-a^i u_{\varepsilon_{i}}^kD_i \varphi+b^i D_i u_{\varepsilon_{i}}^k\varphi+(c-\lambda)u_{\varepsilon_{i}}^k\varphi)\,dx\,dt
$$
$$
=-\int_{t_{k-1}^{(i)}}^{t_k^{(i)}}\int_{\Omega}(g_i D_i \varphi-f\varphi)\,dx\,dt-\int_\Omega u_{\varepsilon_{i}}^{k-1}(t_{k-1}^{(i)}, \cdot) \varphi(t_{k-1}^{(i)},\cdot)\,dx.
$$
Thus, by taking summation with respect to $k$, and then sending $i\rightarrow\infty$, we see that $u$ satisfies \eqref{200508@eq3}.
The proposition is proved.
\end{proof}

%========================================
\subsection{Poincar\'e and embedding inequalities}		\label{S3-2}
%========================================

In this subsection, we
present some inequalities suitable to our problem. The first lemma is a Sobolev-Poincar\'e inequality on small Reifenberg flat domains, proved in \cite{{MR3809039}}.

\begin{lemma}[{\cite[Theorem~3.5]{MR3809039}}]		\label{200708@lem2}
Let $p\in (1,d)$ and $\Omega$ be a  domain in $\bR^d$ satisfying Assumption \ref{A11} (a) with $\gamma\in(0,1/48]$, and let $x_0\in \partial \Omega$ and $R\in (0, R_0/4]$.
Then for any $u\in W^1_p(\Omega_{2R}(x_0))$, we have
$$
\|u-(u)_{\Omega_R(x_0)}\|_{L_{dp/(d-p)}(\Omega_R(x_0))}\lesssim_{d,p} \|Du\|_{L_p(\Omega_{2R}(x_0))}.
$$
\end{lemma}

When $u=0$ on a surface ball, we can remove the average from the left-hand side.
\begin{lemma}	\label{200508@lem1}
Let $\Omega$ be a  domain in $\bR^d$ satisfying Assumption \ref{A11} (a) with $\gamma\in(0,1/48]$.
Let $x_0\in \partial \Omega$, $R\in (0, R_0/4]$, and $u\in W^1_p(\Omega_{2R}(x_0))$. If there exist some $z_0\in \partial \Omega\cap B_R(x_0)$ and $\alpha\in (0,1)$ such that
\begin{equation*}
B_{\alpha R}(z_0)\subset B_R(x_0), \quad u=0 \, \text{ on }\, \partial \Omega\cap B_{\alpha R}(z_0),
\end{equation*}
then
\begin{equation}		\label{200812@A1}
\|u\|_{L_{dp/(d-p)}(\Omega_R(x_0))}\lesssim_{d,p, \alpha} \|Du\|_{L_p(\Omega_{2R}(x_0))},
\end{equation}
provided that $p\in (1,d)$, and
\begin{equation}		\label{201228@eq1}
\|u\|_{L_p(\Omega_R(x_0))}\lesssim_{d,p,\alpha}  R\|Du\|_{L_p(\Omega_{2R}(x_0))},
\end{equation}
provided that $p\in (1, \infty)$.
The same results hold for $x_0\in \Omega$ and $R\in (0, R_0/8]$ with $\Omega_{5R}(x_0)$ in place of $\Omega_{2R}(x_0)$.
\end{lemma}

\begin{proof}
The estimate \eqref{201228@eq1} is a simple consequence of H\"older's inequality and \eqref{200812@A1}, the proof of which is the same as that of \cite[Corollary 3.2 $(a)$]{arXiv:1904.00545}. Here the chain of inclusions
$$
\Omega_R(x_0)\subset \Omega_{2R}(z_0) \subset \Omega_{4R}(z_0) \subset \Omega_{5R}(x_0)
$$
is also used when $x_0\in \Omega$.
\end{proof}

We have the following parabolic Sobolev-Poincar\'e inequalities with mixed norms, the proof of which is based on the embedding results in  \cite[Lemmas 5.3 and 5.4]{arXiv:2007.01986} along with Sobolev-Poincar\'e inequalities for each $x$ and $t$ variables.
We present the proof in Appendix \ref{appendix}.
Note that the usual zero extension technique as in \cite[Corollary~3.2]{arXiv:1904.00545} does not work for the parabolic case, since $u_t$ might not be in $\mathbb{H}^{-1}_{q,p}$ after the extension.

\begin{lemma}		\label{200820@lem2}
Let $\Omega$ be a domain in $\bR^d$ satisfying Assumption \ref{A11} (a) with $\gamma\in(0,1/48]$, and let  $X_0\in \overline{\cQ}$ and $R\in (0, R_0/4]$ such that either
$$
Q_{2R}(X_0)\subset \cQ \quad \text{or}\quad X_0\in \partial \cQ.
$$
If $p,q\in [1,\infty]$, $p_0\in [p,\infty]$, $q_0\in (q,\infty]$, and
$$
\frac{d}{p}+\frac{2}{q}< 1+\frac{d}{p_0}+\frac{2}{q_0},
$$
then for $u\in W^{0,1}_{q,p}(\cQ_{2R}(X_0))$ satisfying
$$
u_t=D_i g_i \quad \text{in }\, \cQ_{2R}(X_0)
$$
in the distribution sense, where $g=(g_1,\ldots, g_d)\in L_{q,p}(\cQ_{2R}(X_0))^d$,
we have
\begin{equation}\label{eqn-200919-1145}
\begin{aligned}
		&\|u-(u)_{\cQ_R(X_0)}\|_{L_{q_0, p_0}(\cQ_R(X_0))}\\
&		\lesssim_{d,p,q,p_0,q_0} R^{1+d/p_0+2/q_0-d/p-2/q} \big(\|Du\|_{L_{q,p}(\cQ_{2R}(X_0))}+\|g\|_{L_{q,p}(\cQ_{2R}(X_0))}\big).		
\end{aligned}
\end{equation}
If we further assume that there exist $Y_0\in \partial \cQ$ and $\alpha\in (0,1)$ such that
$$
u=0 \quad \text{on }\, Q_{\alpha R}(Y_0)\cap \p\cQ, \quad Q_{\alpha R}(Y_0)\subset Q_R(X_0),
$$
then we have
		\begin{equation}\label{eqn-200920-0425}
			\|u\|_{L_{q_0, p_0}(\cQ_R(X_0))}
			\lesssim_{d,p,q,p_0,q_0,\alpha} R^{1+d/p_0+2/q_0-d/p-2/q} \big(\|Du\|_{L_{q,p}(\cQ_{2R}(X_0))}+\|g\|_{L_{q,p}(\cQ_{2R}(X_0))}\big).
		\end{equation}
\end{lemma}

To end this subsection, we prove the following inequality which will be frequently used in our cut-off argument.
\begin{lemma}		\label{200716@lem2}
Let $p\in (1, \infty)$ and Assumption \ref{A11} $(\gamma;m,M)$ be satisfied with $\gamma\in \bigg(0, \frac{1}{160 \sqrt{d+3}}\bigg]$, and let  $X_0\in \Gamma$, $R\in (0, R_0]$, and $\rho\in [R/8, R]$.
Then for any $u\in W^{0,1}_{p}(\cQ_{\rho}(X_0))$ vanishing on $\cD\cap Q_{\rho}(X_0)$, we have
$$
\| u\bI_{A} \|_{L_p(\cQ_{\rho/2}(X_0))}\lesssim_{d,M,p} \gamma \rho \| Du \|_{L_p(\cQ_{\rho}(X_0))},
$$
where
$$
A=\{(s,y): y^1<x_0^1+2\gamma  R, \, y^2>\phi-2\gamma R\}
$$
in the coordinate system associated with $(X_0, R)$.

\end{lemma}

\begin{proof}
By translation we may assume that $X_0=(0,0)$.
Fix the coordinate system associated with the origin and $R$, and
we denote by $\cD_{grid}$ the set of all grid points $z=(\gamma R, k\gamma R)$, where  $k=(k_2,\ldots, k_d)\in \bZ^{d-1}$, such that
$$
z\in \Omega_{\rho/2}, \quad \Omega_{\sqrt{d+3}\gamma R}(z)\cap \{x:x^2>\phi\}\neq \emptyset.
$$
By Lemma \ref{200508@lem1} applied to $\Omega_{2\sqrt{d+3}\gamma R}(z)$, we have
$$
\|u(t,\cdot)\|_{L_p(\Omega_{2\sqrt{d+3}\gamma R}(z))}\lesssim_{d,M,p} \gamma R\|Du(t,\cdot)\|_{L_p(\Omega_{10\sqrt{d+3}\gamma R}(z))}.
$$
Since
$$
\big(\Omega_{\rho/2}\cap A\big) \subset \bigcup_{z\in \cD_{grid}}\Omega_{2\sqrt{d+3}\gamma R}(z)\subset \bigcup_{z\in \cD_{grid}}\Omega_{10\sqrt{d+3}\gamma R}(z)\subset \Omega_{\rho},
$$
we have that
$$
\|u(t,\cdot)\bI_A\|_{L_p(\Omega_{\rho/2})}\lesssim \gamma R\|Du(t,\cdot)\|_{L_p(\Omega_{\rho})},
$$
from which we get the desired estimate.
\end{proof}

%========================================
\subsection{Localization and reverse H\"older's inequality}		
%========================================
In this subsection, we assume that the lower-order coefficients of $\cP$ are all zero, i.e.,
$$
\cP u=-u_t+D_i (a^{ij}D_j u).
$$
We do not impose any regularity assumptions on $a^{ij}$.
Hereafter in this paper, we always assume that $T\in (-\infty, \infty]$.

We first localize the estimate in Proposition \ref{200810@prop1}.
\begin{lemma}		\label{200515@lem9}
Let  $\Omega\subset\bR^d$
and Assumption \ref{ass-0301-2356} be satisfied.
If
$u\in \cH^1_{2,\cD^T,{\rm{loc}}}(\cQ^T)$ satisfies
\begin{equation}		\label{200522@eq1}
\begin{cases}
\cP u = D_ig_i  & \text{in }\, \cQ^T,\\
\cB u = g_in_i& \text{on }\, \cN^T,\\
u = 0 & \text{on }\, \cD^T,
\end{cases}
\end{equation}
where $g=(g_1,\ldots, g_d)\in L_{2,\rm{loc}}(\cQ^T)^d$,
then
for any $X_0\in \overline{\cQ^T}$ and $R\in (0,1]$,  we have
$$
\|Du\|_{L_2(\cQ_{R/2}(X_0))}\lesssim_{d,\Lambda} R^{-1}\|u\|_{L_2(\cQ_R(X_0))}+\|g\|_{L_2(\cQ_R(X_0))}.
$$
\end{lemma}

\begin{proof}
We give a proof which also works when the $L_2$ norms are replaced by the $L_p$ norms provided that the corresponding global estimate is available.
By translation we may assume that $X_0=(0,0)$.
Let
$$
R_k=R(1-2^{-k}), \quad k\in \{1,2,\ldots\}
$$
and $\eta_k$ be an infinitely differentiable function on $\bR^{d+1}$ such that
$$
\begin{gathered}
0\le \eta_k \le 1, \quad \eta_k \equiv1 \, \text{ on }\, \bQ_{R_{k}}, \quad \operatorname{supp}\eta_k \subset \bQ_{R_{k+1}},\\
|(\eta_k)_t|+|D\eta_k|^2\lesssim R^{-2}2^{2k}.
\end{gathered}
$$
Then  $\eta_k u\in \cH^1_{2,\cD^0}(\cQ^0)$ satisfies
\begin{equation}		\label{200831@eq1}
\begin{cases}
\cP (\eta_k u)-\lambda_k (\eta_k u) = D_i g^k_{i}+g^k  & \text{in }\, \cQ^0,\\
\cB (\eta_k u) = g^k_i n_i& \text{on }\, \cN^0,\\
\eta_k u = 0 & \text{on }\, \cD^0,
\end{cases}
\end{equation}
where $\lambda_k\ge \lambda_0$, $\lambda_0=\lambda_0(d,\Lambda)\ge 0$ is from Proposition \ref{200810@prop1},  and
$$
\begin{gathered}		
g^{k}_{i}=\eta_k g_i +\sum_{j=1}^da^{ij}D_j \eta_k u,\\
g^k=-\sum_{i=1}^dD_i \eta_k g_i -(\eta_k)_t u+\sum_{i,j=1}^da^{ij}D_j uD_i \eta_k-\lambda_k(\eta_k u).
\end{gathered}
$$
By \eqref{est-0302-0000} applied to \eqref{200831@eq1}, we have
$$
\begin{aligned}
&\|D(\eta_k u)\|_{L_2(\cQ^0)}+\sqrt{\lambda_k} \|\eta_k u\|_{L_2(\cQ^0)}\lesssim \Bigg(\frac{2^k}{R} +\frac{2^{2k}}{R^2\sqrt{\lambda_k}}+\sqrt{\lambda_k}\Bigg)\|u\|_{L_2(\cQ_{R})}\\
&\quad +\Bigg(1+\frac{2^k}{R\sqrt{\lambda_k}}\Bigg)\|g\|_{L_2(\cQ_{R})}
+\frac{2^k}{R\sqrt{\lambda_k}} \|D(\eta_{k+1}u)\|_{L_2(\cQ_0)}.
\end{aligned}
$$
Set
$$
\cU_k=\|D(\eta_k u)\|_{L_2(\cQ^0)}, \quad \cU=\|u\|_{L_2(\cQ_R)},\quad \cF=\|g\|_{L_2(\cQ_{R})}.
$$
By multiplying both sides of the above inequality by $\varepsilon^k$ and summing the terms respect to $k$, we obtain
$$
\begin{aligned}
\sum_{k=1}^\infty \varepsilon^k \cU_k
&\le C\cU\sum_{k=1}^\infty \varepsilon^k \Bigg(\frac{2^k}{R} +\frac{2^{2k}}{R^2\sqrt{\lambda_k}}
+\sqrt{\lambda_k}\Bigg)\\
&\quad +C\cF \sum_{k=1}^\infty \varepsilon^k \Bigg(1+\frac{2^k}{R\sqrt{\lambda_k}}\Bigg)
+C_0 \sum_{k=1}^\infty \frac{ \varepsilon^k 2^k}{R\sqrt{\lambda_k}} \cU_{k+1},
\end{aligned}
$$
where we may assume $C_0\ge (\lambda_0+1)^{1/2}$, and each summation is finite upon choosing
$$
\varepsilon=1/4, \quad \lambda_k=\Bigg(\frac{C_02^{k+2}}{R}\Bigg)^2.
$$
Indeed we have
$$
\sum_{k=1}^\infty \varepsilon^k \Bigg(\frac{2^k}{R} +\frac{2^{2k}}{R^2\sqrt{\lambda_k}}
+\sqrt{\lambda_k}\Bigg)\le \frac{C}{R},
$$
$$
\sum_{k=1}^\infty \varepsilon^k \Bigg(1+\frac{2^k}{R\sqrt{\lambda_k}}\Bigg)\le C,
$$
$$
C_0\sum_{k=1}^\infty \frac{\varepsilon^k2^k}{R\sqrt{\lambda_k}} \cU_{k+1}=\sum_{k=1}^\infty \varepsilon^{k+1}\cU_{k+1}=\sum_{k=2}^\infty \varepsilon^k \cU_k.
$$
Therefore,
$$
\sum_{k=1}^\infty   \varepsilon^k \cU_k\le \frac{C}{R}\cU+C \cF+\sum_{k=2}^\infty \varepsilon^k \cU_k,
$$
which implies the desired estimate.
\end{proof}

\begin{lemma}		\label{200622@lem1}
Let $\alpha\in (0,1)$, $q_1\in \big(\frac{2(d+2)}{d+4}, 2\big)$, and $\gamma\in \big(0,\frac{1}{48}\big]$.
If  Assumptions \ref{A11} (a) and \ref{ass-0301-2356} are satisfied,
and if for any $X\in \Gamma$ and $\rho\in (0, R_0/4]$, there exists $Y\in \partial \cQ$ such that
\begin{equation}		\label{201224@eq1}
Q_{\alpha \rho}(Y)\cap \partial \cQ\subset \cD, \quad Q_{\alpha \rho}(Y)\subset Q_\rho(X),
\end{equation}
then the following assertion holds.
Let  $u\in \cH^1_{2,\cD^T,{\rm{loc}}}(\cQ^T)$ satisfy \eqref{200522@eq1} with $g=(g_1,\ldots, g_d)\in L_{2,\rm{loc}}(\cQ^T)^d$.
Then for any $X_0\in \overline{\cQ^T}$ and $R\in (0, R_0]$ satisfying either
$$
Q_{R/2}(X_0)\subset \cQ \quad \text{or}\quad X_0\in \partial \cQ,
$$
we have
$$
(|Du|^2)^{1/2}_{\cQ_{R/16}(X_0)}\lesssim_{d,\Lambda, \alpha, q_1} (|Du|^{q_1})^{1/q_1}_{\cQ_{R}(X_0)}+(|g|^2)^{1/2}_{\cQ_R(X_0)}.
$$
\end{lemma}

\begin{proof}
By translation, we may assume that $X_0=(0,0)$.
In the case when $Q_{R/2}\subset \cQ$, by Lemma \ref{200515@lem9} applied to $u-(u)_{Q_{R/4}}$, we have
$$
(|Du|^2)^{1/2}_{Q_{R/8}}\lesssim R^{-1}(|u-(u)_{Q_{R/4}}|^2)^{1/2}_{Q_{R/4}}+(|g|^2)^{1/2}_{Q_{R/4}},
$$
from which together with \eqref{eqn-200919-1145}, we get the desired estimate.
In the case when $X_0\in \partial \cQ$, we consider the following two cases:
$$
Q_{R/4}\cap \Gamma\neq \emptyset, \quad Q_{R/4}\cap \Gamma=\emptyset.
$$

\begin{enumerate}[(i)]
\item
$Q_{R/4}\cap \Gamma\neq \emptyset$.
Due to Lemma \ref{200515@lem9}, it suffices to show that
\begin{equation}		\label{200807@eq6}
R^{-1}(|u|^2)^{1/2}_{\cQ_{R/2}}\lesssim (|Du|^{q_1})^{1/q_1}_{\cQ_{R}}+(|g|^2)^{1/2}_{\cQ_{R}}.
\end{equation}
We aim to apply Lemma \ref{200820@lem2}. For this, we first fix some $X \in Q_{R/4}\cap\Gamma$. By the assumption, we can find some $Y_0\in\p\Omega$ satisfying \eqref{201224@eq1} with $\rho=R/4$. Noting $u$ vanishes on $\cQ_{\alpha R/4}(Y_0)\subset \cQ_{R/4}(X)\subset \cQ_{R/2}$, by \eqref{eqn-200920-0425} with $R/2$ in place of $R$, we conclude \eqref{200807@eq6}.
\item
$Q_{R/4}\cap \Gamma=\emptyset$.
In this case, we have either
$$
\big(Q_{R/4}\cap \partial \cQ\big)\subset \cD\quad \text{or}\quad \big(Q_{R/4}\cap \partial \cQ\big)\subset \cN,
$$
where the proof of the first case is the same as in $(i)$ and the second case is the same as the interior case.
\end{enumerate}
The lemma is proved.
\end{proof}

From Lemma \ref{200622@lem1} along with Gehring's lemma and Agmon's idea, we conclude the following reverse H\"older's inequality. For a given constant $\lambda>0$ and functions $u$, $f$, and  $g=(g_1,\ldots, g_d)$, we write
$$
U=|Du|+\sqrt{\lambda} |u|, \quad F=
|g|+\frac{|f|}{\sqrt{\lambda}}.
$$
We also denote for a function $v$ defined on a domain $Q$ that
$$
\overline{v}=v \bI_Q.
$$

\begin{lemma}		\label{200710@lem5}
Let $\alpha\in (0,1)$,  $p>2$,  and $\gamma\in \big(0, \frac{1}{48}\big]$.
If Assumptions \ref{A11} (a) and \ref{ass-0301-2356} are satisfied,
and if for any $X\in \Gamma$ and $\rho\in (0, R_0/4]$, there exists $Y\in \partial \cQ$ such that \eqref{201224@eq1} holds,
then we have the following.
Let $u\in \cH^1_{2,\cD^T,\rm{loc}}(\cQ^T)$ satisfy
\begin{equation}		\label{200722@A1}
\begin{cases}
\cP u - \lambda u= D_ig_i +f & \text{in }\, \cQ^T,\\
\cB u = g_in_i& \text{on }\, \cN^T,\\
u = 0 & \text{on }\, \cD^T,
\end{cases}
\end{equation}
where  $\lambda>0$, $g=(g_1,\ldots, g_d)\in L_{p,{\rm{loc}}}(\cQ^T)^d$, and $f\in L_{p,{\rm{loc}}}(\cQ^T)$.
There exists a constant  $p_0\in (2, p)$ depending only on $d$, $\Lambda$, $\alpha$, and $p$, such that for any $X_0\in \bR^{d+1}$ and $R\in (0, R_0]$,  we have
$$
\big(\overline{U}^{p_0}\big)^{1/p_0}_{Q_{R/2}(X_0)}\lesssim_{d,\Lambda,\alpha,p} \big(\overline{U}^2\big)^{1/2}_{Q_R(X_0)}+\big(\overline{F}^{p_0}\big)^{1/p_0}_{Q_R(X_0)}.
$$
The same result holds with $|Du|$ and $|g|$ in place of $U$ and $F$, respectively, provided that $\lambda=0$ and $f\equiv 0$.
\end{lemma}

\begin{proof}
We first prove the lemma for $\lambda=0$ and $f\equiv 0$.
More precisely, we show that  there exists $p_1\in (2,p]$, depending only on $d$, $\Lambda$, $\alpha$, and $p$, such that for any $p_0\in (2,p_1]$, $X_0\in \bR^{d+1}$, and $R\in (0, R_0]$, we have
\begin{equation}		\label{200717@B1}
\big(|\overline{Du}|^{p_0}\big)_{Q_{R/2}(X_0)}^{1/p_0}\lesssim_{d,\Lambda, \alpha, p} \big(|\overline{Du}|^{2}\big)_{Q_{R}(X_0)}^{1/2}+\big(|\overline{g}|^{p_0}\big)_{Q_{R}(X_0)}^{1/p_0}.
\end{equation}
Fix a number $q_1\in \big(\frac{2(d+2)}{d+4},2\big)$.
Then by Lemma \ref{200622@lem1} with a covering argument, we have that for $X_0\in \bR^{d+1}$ and $R\in (0, R_0]$,
$$
\dashint_{Q_{R/2}(X_0)} \Phi^{2/q_1}\,dX\lesssim_{d,\Lambda, \alpha} \Bigg(\dashint_{Q_R(X_0)}\Phi\,dX\Bigg)^{2/q_1}+\dashint_{Q_R(X_0)}\Psi^{2/q_1}\,dX,
$$
 where $\Phi=|\overline{Du}|^{q_1}$ and  $\Psi=|\overline{g}|^{q_1}$.
Thus by Gehring's lemma (see, for instance, \cite[Ch.V]{MR0717034}), we get the desired result \eqref{200717@B1}.

For the case when $\lambda>0$, we use an idea by S. Agmon.
Denote by $(t,z)=(t,x,\tau)$ a point in $\bR^{d+2}$, where $z=(x,\tau)\in \bR^{d+1}$.
We define
$$
\hat{u}(t,z)=u(t,x)\eta(\tau), \quad \eta(\tau)=\cos(\sqrt{\lambda}\tau+\pi/4),
$$
$$
\hat{\cQ}=\{(t,z): t\in \bR,\, z\in \Omega\times \bR\}, \quad \hat{\cQ}^T=\{(t,z)\in \hat{\cQ}: t<T\},
$$
$$
\hat{\cD}=\{(t,z):t\in \bR, \, z\in \cD\times \bR\}, \quad \hat{\cD}^T=\{(t,z)\in \hat{\cD}: t<T\},
$$
$$
\hat{\cN}=\{(t,z):t\in \bR, \, z\in \cN\times \bR\}, \quad \hat{\cN}^T=\{(t,z)\in \hat{\cN}: t<T\},
$$
$$
\hat{Q}_r(t_0,z_0)=(t_0-r^2,t_0)\times \{z\in \bR^{d+1}: |z-z_0|<r\}.
$$
Since $u$ satisfies \eqref{200722@A1}, we see  that  $\hat{u}\in \cH^1_{2,\hat{\cD}^T, \rm{loc}}(\hat{\cQ}^T)$ satisfies
\begin{equation}		\label{200722@eq1}
\begin{cases}
\cP \hat{u}+D_{\tau\tau}\hat{u} = {\displaystyle \sum_{i=1}^dD_i (g_i \eta)} + f\eta & \text{in }\, \hat{\cQ}^T,\\
\cB \hat{u} = {\displaystyle \sum_{i=1}^d g_i\eta n_i}& \text{on }\, \hat{\cN}^T,\\
\hat{u} = 0 & \text{on }\, \hat{\cD}^T.
\end{cases}
\end{equation}
Since $n_\tau=0$, the operator $\cB$ is also the conormal derivative operator associated with $\cP+D_{\tau\tau}$.
Note that the coefficients of $\cP+D_{\tau\tau}$ satisfy the ellipticity condition with the same constant $\Lambda$, and that $\hat{\cQ}$, $\hat{\cD}$, and $\hat{\cN}$ satisfy the same conditions of the lemma.
Moreover, we can rewrite the source term in the divergence form since
$$
f\eta=D_\tau(f \hat{\eta}), \quad \text{where }\, \hat{\eta}(\tau)=\frac{\sin(\sqrt{\lambda}\tau+\pi/4)}{\sqrt{\lambda}}.
$$
Hence, we can apply the above result for $\lambda=0$ to \eqref{200722@eq1} to see that there exists $p_2\in (2,p]$ determined by $d+1$, $\Lambda$, $\alpha$, and $p$, such that
for any $p_0\in (2,p_2]$, $(t_0,z_0)\in \bR^{d+2}$, and $R\in (0, R_0]$, 	
\begin{align}
\nonumber
&\big(|\overline{D_{z}\hat{u}}|^{p_0}\big)_{\hat{Q}_{R/2}(t_0,z_0)}^{1/p_0}\\
\nonumber
&\lesssim \big(|\overline{D_{z}\hat{u}}|^{2}\big)_{\hat{Q}_{R}(t_0,z_0)}^{1/2}+\big(|\overline{g_i\eta}|^{p_0}\big)^{1/p_0}_{\hat{Q}_R(t_0,z_0)}+\big(|\overline{f\hat{\eta}}|^{p_0}\big)^{1/p_0}_{\hat{Q}_R(t_0,z_0)}\\
\label{200724@B1}
&\lesssim_{d,\Lambda, M, p}  \big(|\overline{D_{z}\hat{u}}|^{2}\big)_{\hat{Q}_{R}(t_0,z_0)}^{1/2}+\big(\overline{F}^{p_0}\big)^{1/p_0}_{Q_R(t_0,z_0)}.
\end{align}
Now we fix $p_0$ such that $2<p_0\le \min\{p_1,p_2\}$, where $p_1$ is the constant from the case $\lambda=0$.
Then by \eqref{200724@B1} with $z_0=(x_0,0)\in \bR^{d+1}$ and the fact that
$$
\dashint^{-R/4}_{R/4} |\eta|^{p_0}\,d\tau\gtrsim_{p_0}1, \quad |\overline{Du}(t,x)\eta(\tau)|\le |\overline{D_z\hat{u}}(t,z)|\le \overline{U}(t,x),
$$
we have
$$
\big(|\overline{Du}|^{p_0}\big)^{1/p_0}_{Q_{R/4}(t_0,x_0)}
\lesssim \big(|\overline{D_{z}\hat{u}}|^{p_0}\big)_{\hat{Q}_{R/2}(t_0,z_0)}^{1/p_0}\lesssim \big(\overline{U}^2\big)^{1/2}_{Q_R(t_0,x_0)}+\big(\overline{F}^{p_0}\big)^{1/p_0}_{Q_R(t_0,x_0)}.
$$
Similarly, using the fact that
$$
\dashint^{-R/4}_{R/4} |\eta'|^{p_0}\,d\tau\gtrsim_{p_0}\lambda^{p_0/2}, \quad |\overline{u}(t,x)\eta'(\tau)|\le |\overline{D_z\hat{u}}(t,z)|\le \overline{U}(t,x),
$$
we get
$$
\big(\big|\sqrt{\lambda} \overline{u}\big|^{p_0}\big)^{1/p_0}_{Q_{R/4}(t_0,x_0)}
\lesssim \big(|\overline{D_{z}\hat{u}}|^{p_0}\big)_{\hat{Q}_{R/2}(t_0,z_0)}^{1/p_0}\lesssim \big(\overline{U}^2\big)^{1/2}_{Q_R(t_0,x_0)}+\big(\overline{F}^{p_0}\big)^{1/p_0}_{Q_R(t_0,x_0)}.
$$
Combining these together and using a covering argument, we get the desired estimate.
The lemma is proved.
\end{proof}

%========================================
\section{Estimates with time-independent separation and constant coefficients}        \label{sec4}
%========================================
In this section, we derive various estimates when the separation is time-independent and the coefficients are all constants. Throughout the section, we consider
$$
\cP_0 u=-u_t+D_i(a^{ij}_0D_j u),
$$
where the coefficients $a^{ij}_0$ are constants.
The corresponding conormal derivative of $u$ is denoted by $\cB_0 u$.
%========================================
\subsection{Estimates near curved boundary}
%========================================
Recall that
$$
U=|Du|+\sqrt{\lambda}|u|,
$$
and set
$$
U_t=|Du_t|+\sqrt{\lambda}|u_t|.
$$
Compared to Lemma \ref{200710@lem5}, in the following proposition we further estimate some time derivatives. Here, $\Gamma$ is assumed to be time-independent, which clearly satisfies Assumption \ref{ass-0301-2356}, and $\cP$ is replaced with the constant-coefficient operator $\cP_0$.
\begin{proposition}		\label{200825@prop1}
Let $\alpha\in (0,1)$ and $\gamma\in\big(0, \frac{1}{48}\big]$.
Suppose that the spatial domain $\Omega$ satisfies Assumption \ref{A11} (a),  the separation $\Gamma$ between $\cD$ and $\cN$ is time-independent, and for any $X\in \Gamma$ and $\rho\in (0, R_0/4]$, there exists $Y\in \partial \cQ$ such that  \eqref{201224@eq1} holds.
Let $(0,0)\in \Gamma$, $R\in (0, R_0]$, and   $u\in W^{0,1}_2(\cQ_R)$ satisfy
\begin{equation}		\label{200825@eq1}
\begin{cases}
\cP_0 u-\lambda u =0  & \text{in }\, \cQ_R,\\
\cB_0 u = 0& \text{on }\, Q_R\cap \cN,\\
u = 0 & \text{on }\, Q_R\cap \cD,
\end{cases}
\end{equation}
where $\lambda\ge 0$.
Then $u_t\in W^{0,1}_2(\cQ_{R/2})$ and
$$
R^{d+2-\frac{d+2}{p_0}}\|U\|_{L_{p_0}(\cQ_{R/2})}+R^{d+4-\frac{d+2}{p_0}}\|U_t\|_{L_{p_0}(\cQ_{R/2})}\lesssim_{d, \Lambda,\alpha}\|U\|_{L_1(\cQ_R)},
$$
where $p_0=p_0(d,\Lambda, \alpha)\in (2,4)$ is from Lemma \ref{200710@lem5} with $g_i=f=0$.
\end{proposition}

\begin{remark}
In the  proposition above and throughout the paper, $u\in W^{0,1}_{p}(\cQ_R)$ is said to satisfy \eqref{200825@eq1} if $u$ vanishes on $Q_R\cap \cD$ and
$$
\int_{\cQ_R} u \varphi_t\,dX-\int_{\cQ_R}a^{ij}_0D_ju  D_i\varphi\,dX-\lambda \int_{\cQ_R} u \varphi \,dX=0
$$
for any $\varphi\in W^{1,1}_{p/(p-1)}(\cQ_R)$ that vanishes on $\partial \cQ_R\setminus \cN$.
\end{remark}

The rest of this subsection is devoted to the proof of Proposition \ref{200825@prop1}.

\begin{lemma}		\label{200825@lem2}
Under the same conditions in Proposition \ref{200825@prop1} with $\lambda=0$,
we have that for $r\in (0, R)$,
$u_t\in W^{1,1}_2(\cQ_r)$
and
$$
(R-r)\|u_t\|_{L_{2}(\cQ_{r})}+(R-r)^2\|Du_t\|_{L_{2}(\cQ_{r})}+(R-r)^3\|u_{tt}\|_{L_2(\cQ_r)}\lesssim_{d, \Lambda}\|Du\|_{L_2(\cQ_R)}.
$$
\end{lemma}

\begin{proof}
By a standard mollification technique, it suffices to prove the desired estimate under the assumption that $u$ and $Du$ are smooth with respect to $t$.
For $\rho, \tau$ with $0<\rho<\tau\le R$, let $\eta_{\rho,\tau}$ be an infinitely differentiable  function in $\bR^{d+1}$ such that
$$
0\le \eta_{\rho, \tau}\le 1, \quad \eta_{\rho,\tau}\equiv 1 \,\text{ in }\, \bQ_\rho, \quad \operatorname{supp}\eta_{\rho,\tau}\subset \bQ_\tau,
$$
$$
|D \eta_{\rho, \tau}|^2+|(\eta_{\rho, \tau})_t|\lesssim_d (\tau-\rho)^{-2}.
$$
Since the boundary portions in \eqref{200825@eq1} are time-independent, we can apply  $\eta_{\rho, \tau}^2 u$ as a test function to get the usual Caccioppoli inequality
$$
\|Du\|_{L_2(\cQ_\rho)}\lesssim_{d,\Lambda}  (\tau-\rho)^{-1}\|u\|_{L_2(\cQ_\tau)}.
$$
Since $u_t$ satisfies the same equation and boundary conditions, we also have
\begin{equation}		\label{200825@eq1a}
\|Du_t\|_{L_2(\cQ_\rho)}\lesssim (\tau-\rho)^{-1}\|u_t\|_{L_2(\cQ_\tau)}.
\end{equation}
Let
$$
\rho\le \rho_1<\rho_2\le \tau, \quad \rho_0=\frac{\rho_1+\rho_2}{2}.
$$
We test \eqref{200825@eq1} again with $\eta_{\rho_1,\rho_0}^2u_t$ to get
$$
\begin{aligned}
\int_{\cQ_{\rho_0}} |\eta_{\rho_1, \rho_0} u_t|^2\,dX
&=-\int_{\cQ_{\rho_0}} \eta_{\rho_1,\rho_0}^2a_0^{ij}D_j uD_i u_t\,dX\\
&\quad -2\int_{\cQ_{\rho_0}} \eta_{\rho_1, \rho_0}D_i \eta_{\rho_1, \rho_0} a_0^{ij} D_j u u_t\,dX.
\end{aligned}
$$
By  Young's inequality and \eqref{200825@eq1a} with $\rho_0$ and $\rho_2$ in place of $\rho$ and $\tau$, respectively, we have
$$
\int_{\cQ_{\rho_1}}|u_t|^2\,dX\le \frac{C}{(\rho_2-\rho_1)^2}\int_{\cQ_{\rho_2}}|Du|^2\,dX+\varepsilon \int_{\cQ_{\rho_2}} |u_t|^2\,dX
$$
for $\varepsilon\in (0,1)$, where $C=C(d, \varepsilon)>0$.
Since the above inequality holds for any $\rho_1,\rho_2$ with $\rho \le \rho_1<\rho_2\le \tau$, by a standard iteration argument, we get
\begin{equation}\label{eqn-201214-1038}
	\|u_t\|_{L_2(\cQ_{\rho})}\lesssim (\tau-\rho)^{-1} \|Du\|_{L_2(\cQ_\tau)}.
\end{equation}
From \eqref{200825@eq1a} with $\tau$ replaced by $(\rho+\tau)/2$ and \eqref{eqn-201214-1038} with $\rho$ replaced by $(\rho+\tau)/2$, we obtain
\begin{equation}\label{eqn-201214-1036}
	\|Du_t\|_{L_2(\cQ_\rho)}\lesssim (\tau-\rho)^{-1}\|u_t\|_{L_2(\cQ_{(\rho+\tau)/2})} \lesssim (\tau-\rho)^{-2} \|Du\|_{L_2(\cQ_\tau)}.
\end{equation}
Again, since we can differentiate the problem in $t$, the $u_{tt}$ estimate follows from \eqref{eqn-201214-1038} with $u$ replaced by $u_t$ and \eqref{eqn-201214-1036}, with parameters chosen properly
\begin{equation*}
	\|u_{tt}\|_{L_2(\cQ_{\rho})}\lesssim (\tau-\rho)^{-1} \|Du_t\|_{L_2(\cQ_{(\rho+\tau)/2})} \lesssim (\tau-\rho)^{-3} \|Du\|_{L_2(\cQ_\tau)}
\end{equation*}
The lemma is proved.
\end{proof}

We are ready to prove Proposition \ref{200825@prop1}.

\begin{proof}[Proof of Proposition \ref{200825@prop1}]
Due to Agmon's idea (cf. the proof of Lemma \ref{200710@lem5}), it suffices to prove the proposition for $\lambda=0$.
By a rescaled version of Lemma \ref{200710@lem5} and a covering argument, we see that for $0<\rho<\tau\le R$,
\begin{equation}		\label{200831@A1}
\|Du\|_{L_{p_0}(\cQ_\rho)}\lesssim (\tau-\rho)^{\frac{d+2}{p_0}-\frac{d+2}{2}} \|Du\|_{L_2(\cQ_\tau)}.
\end{equation}
From H\"older's and Young's inequalities, we have
\begin{equation}		\label{200825@eq3}
\|Du\|_{L_{p_0}(\cQ_\rho)}\le \varepsilon \|Du\|_{L_{p_0}(\cQ_\tau)}
+C\varepsilon^{-\frac{p_0}{p_0-2}} (\tau-\rho)^{(d+2)(1/p_0-1)}
\|Du\|_{L_{1}(\cQ_\tau)},
\end{equation}
for any $\varepsilon>0$, where $C=C(d,\Lambda, \alpha, \varepsilon)$.
Set
$$
\tau_k=\frac {3R} 4+\frac R 4 (1-2^{-k}), \quad k\in \{0,1,2,\ldots\}.
$$
Then by \eqref{200825@eq3} with $\tau_{k}$ and $\tau_{k+1}$ in place of $\rho$ and $\tau$, respectively, we get
$$
\|Du\|_{L_{p_0}(\cQ_{\tau_k})}\le \varepsilon \|Du\|_{L_{p_0}(\cQ_{\tau_{k+1}})}+C \varepsilon^{-\frac{p_0}{p_0-2}}
R^{(d+2)(1/p_0-1)}
2^{(d+2)(1-1/p_0)k}\|Du\|_{L_{1}(\cQ_{\tau_{k+1}})}.
$$
Multiplying both sides of the above inequality by $\varepsilon^k$ and summing the terms with respect to $k=0,1,\ldots$, we obtain that
$$
\begin{aligned}
\sum_{k=0}^\infty \varepsilon^k \|Du\|_{L_{p_0}(\cQ_{\tau_k})}
&\le \sum_{k=1}^\infty \varepsilon^k \|Du\|_{L_{p_0}(\cQ_{\tau_k})}\\
&\quad +C\varepsilon^{-\frac{p_0}{p_0-2}}R^{(d+2)(1/p_0-1)}
\sum_{k=0}^\infty \big(\varepsilon2^{(d+2)(1 -1/{p_0})}\big)^k\|Du\|_{L_{1}(\cQ_{\tau_{k+1}})},
\end{aligned}
$$
where each summation is finite upon choosing $\varepsilon=2^{-1-(d+2)(1-1/{p_0})}$.
This yields
\begin{equation}		\label{200825@A1}
\|Du\|_{L_{p_0}(\cQ_{3R/4})}\lesssim R^{\frac{d+2}{p_0}-(d+2)}\|Du\|_{L_1(\cQ_R)}.
\end{equation}
Since $u_t$ satisfies the same equation and the boundary conditions,
by \eqref{200825@A1}, H\"older's inequality, and Lemma \ref{200825@lem2}, we have
$$
\|Du_t\|_{L_{p_0}(\cQ_{R/2})}\lesssim R^{\frac{d+2}{p_0}-\frac{d+2}{2}}\|Du_t\|_{L_2(\cQ_{2R/3})}\lesssim R^{\frac{d+2}{p_0}-\frac{d+2}{2}-2}\|Du\|_{L_2(\cQ_{3R/4})}\lesssim R^{\frac{d+2}{p_0}-d-4}\|Du\|_{L_1(\cQ_{R})}.
$$
The proposition is proved.
\end{proof}

%========================================
\subsection{Estimates near flat boundary}
%========================================

Now, we prove that the problem that we perturb from -- the problem with constant coefficients, cylindrical separation, and flat boundary, can reach the optimal regularity $Du\in L^{x'}_pL^{(t,x'')}_{\infty}$. In this subsection, we additionally assume that $(a_0^{ij})$ is symmetric and set
$$
Q_R^+=Q_R\cap \{(t,x): x_1>0\},
$$
$$
D=B_1\cap \{x:x_1=0, \, x_2>\phi\},
$$
$$
N=B_1\cap  \{x:x_1=0, \, x_2<\phi\},
$$
where
$$
\phi=\phi(x^3,\ldots, x^{m+2}):\bR^{m}\to \bR
$$
 is a Lipschitz function of $m$ variables with Lipschitz constant $M$ and satisfying $\phi(0,\ldots,0)=0$.
Here, $m\in \{0,1,\ldots, d-2\}$ and if $m=0$, then $\phi$ is understood as the constant function $\phi\equiv 0$.
We write $X=(t,x)=(t,x',x'')\in  \bR^{d+1}$, where
$$
x'=(x^1,\ldots, x^{m+2})\in \bR^{m+2}, \quad x''=(x^{m+3},\ldots, x^d)\in \bR^{d-2-m}.
$$

\begin{proposition}\label{200826@prop1}
If $u\in W^{0,1}_2(Q_1^+)$ satisfy
$$
\begin{cases}
\cP_0 u - \lambda u=0 &\quad \text{in }\, Q_1^+,\\
\cB_0 u=0 &\quad \text{on }\, (-1,0)\times N\\
u=0 &\quad \text{on }\, (-1,0)\times D,
\end{cases}
$$
where $\lambda\ge0$, then for $p\in \big[2, \frac{2(m+2)}{m+1}\big)$ we have that
$$
U\in L^{x'}_pL^{(t,x'')}_\infty(Q_{1/2}^+)
$$
and
$$
\|U\|_{L^{x'}_pL^{(t,x'')}_{\infty}(Q_{1/2}^+)}\lesssim_{d,\Lambda,M,p} \|U\|_{L_1(Q_1^+)}.
$$
\end{proposition}

The rest of this subsection is devoted to the proof of Proposition \ref{200826@prop1}.

\begin{lemma}		\label{200320@lem1}
Under the same conditions in Proposition \ref{200826@prop1} with $\lambda=0$,
we have that for $0<r<R\le 1$,
$$
u\in W^{1,1}_2(Q_r^+), \quad u_t, D_k u\in W^{0,1}_2(Q_r^+), \quad k\in \{m+3,\ldots, d\}
$$
and
\begin{equation}		\label{201105@eq1}
\|u_t\|_{L_2(Q_r^+)}+\|Du_t\|_{L_2(Q_r^+)}+\|DD_{x''}u\|_{L_2(Q_r^+)}\lesssim_{d,\Lambda, r, R}\|Du\|_{L_2(Q_R^+)}.
\end{equation}
Generally, for $i\in \{0,1,2,\ldots\}$,  we have
\begin{equation}		\label{201105@eq2}
\|DD^i_{(t,x'')}u\|_{L_2(Q_r^+)}\lesssim_{d,\Lambda, r, R, i}\|Du\|_{L_2(Q_R^+)}.
\end{equation}
\end{lemma}

\begin{proof}
By the same argument used in the proof of Lemma \ref{200825@lem2}, we obtain the bounds of the first two terms in \eqref{201105@eq1}.
The bound of the last term in \eqref{201105@eq1} and also the inequality \eqref{201105@eq2} follow from the Caccioppoli inequality and  the fact that members of $D^i_{(t,x'')}u$ satisfy the same equation and the boundary conditions in $Q_{R'}^+$ for any $R'<R$.
We omit the details.
\end{proof}

We are now ready to present the proof of Proposition \ref{200826@prop1}.

\begin{proof}[Proof of Proposition \ref{200826@prop1}]
Due to Agmon's idea,  it suffices to prove the proposition for $\lambda=0$.
Also by a change of variables, we may assume that $a_0^{ij}=\delta_{ij}$.

We first claim that
\begin{equation}		\label{200831@A2}
\|Du\|_{L^{x'}_pL^{(t,x'')}_{\infty}(Q_{1/2}^+)}\lesssim_{d,\Lambda,M,p} \|Du\|_{L_2(Q_1^+)}.
\end{equation}
For almost every $(t,x'')$ with $-1<t<0$ and $|x''|<1$,
the function $v=u(t, \cdot, x'')$ satisfies
$$
\begin{cases}
\Delta_{x'} v= f & \text{in }\, B'^+_{3/4},\\
D_1 v = 0& \text{on }\, N',\\
v = 0 & \text{on }\, D',
\end{cases}
$$
where we set
$$
f=u_t(t,\cdot,x'')-\Delta_{x''} u(t,\cdot,x''),
$$
$$
B'_{3/4}=\{x'\in \bR^{m+2}: |x'|<{3/4}\},\quad B'^+_{3/4}=B_{3/4}'\cap \{x'\in \bR^{m+2}:x_1>0\},
$$
$$
D'=B'_{3/4}\cap \{x'\in \bR^{m+2}: x_1=0, x_2>\phi\},
$$
$$
N'=B'_{3/4}\cap \{x'\in \bR^{m+2}: x_1=0, x_2<\phi\}.
$$
Notice from Lemma \ref{200320@lem1} that $f\in L_2(B'^+_{3/4})$.
Thus by a rescaled version of \cite[Lemma 3.2]{arXiv:2003.10980} with a change of variables, we have
$$
\|D_{x'}v\|_{L_p(B'^+_{1/2})}\lesssim_{\Lambda, M,p}\|D_{x'}v\|_{L_2(B'^+_{3/4})}+\|f\|_{L_2(B'^+_{3/4})}.
$$
Taking $L_2$ norm in $\{(t,x''):-1<t<0, |x''|<1/2\}$ and using Lemma \ref{200320@lem1}, we obtain
$$
\|Du\|_{L_2^{(t,x'')}L_p^{x'}(Q_{1/2}^+)}\lesssim \|Du\|_{L_2(Q_1^+)},
$$
and thus by Minkowski's inequality, we have
$$
\|Du\|_{L_p^{x'}L_2^{(t,x'')}(Q_{1/2}^+)}\lesssim \|Du\|_{L_2(Q_1^+)}.
$$
From the above inequality with scaling and the fact that the members of $D^i_{(t,x'')}u$ satisfy the same equation and the boundary condition in any smaller cylinder $Q_{R'}^+$ with $R'<1$, we also have
$$
\|DD_{(t,x'')}^i u\|_{L_p^{x'}L_2^{(t,x'')}(Q_{1/2}^+)}\lesssim_{d, \Lambda, M, p, i}\|Du\|_{L_2(Q_1^+)}
$$
for $i\in \{1,2,\ldots\}$, where we also used Lemma \ref{200320@lem1}.
Therefore, by the Sobolev embedding in $(t,x'')$, there exists $k$ such that
$$
\|D u\|_{L_p^{x'}L_\infty^{(t,x'')}(Q_{1/2}^+)}\lesssim \sum_{i=0}^k \|DD^i_{(t,x'')}u\|_{L_{p}^{x'}L_2^{(t,x'')}(Q_{1/2}^+)} \lesssim \|Du\|_{L_2(Q_1^+)},
$$
which implies the claim \eqref{200831@A2}.

It is well known that \eqref{200831@A2} holds (with $p=\infty$) when  $u$ satisfies purely Dirichlet/conormal derivative boundary conditions.
The corresponding interior estimate also holds.
Hence, by scaling and a covering argument, one can see that
for $0<\rho<r\le 1$,
$$
\|Du\|_{L^{x'}_pL^{(t,x'')}_{\infty}(Q_{\rho}^+)}\lesssim  (\tau-\rho)^{\frac{m+2}{p}-\frac{d+2}{2}}\|Du\|_{L_2(Q_r^+)},
$$
which corresponds to \eqref{200831@A1}.
Similarly as in the proof of Proposition \ref{200825@prop1}, we conclude that
$$
\|Du\|_{L^{x'}_pL^{(t,x'')}_{\infty}(Q_{1/2}^+)}\lesssim\|Du\|_{L_1(Q_1^+)}.
$$
This proves the proposition for $\lambda=0$.
The proposition is proved.
\end{proof}

%========================================
\section{Proof of Theorem \ref{MT1}}		\label{S5}
%========================================

This section is devoted to the proof of Theorem \ref{MT1}.

\subsection{Decomposition}		\label{S5-1}
~
In this subsection, we consider the operator $\cP$ without lower-order terms, i.e.,
$$
\cP u=-u_t+D_i (a^{ij}D_j u).
$$
Recall that
$$
U=|Du|+\sqrt{\lambda} |u|, \quad F=|g|+\frac{|f|}{\sqrt{\lambda}}.
$$

\begin{proposition}		\label{200726@prop1}
Let
$$
p>2, \quad \gamma\in \Bigg(0, \frac{1}{160\sqrt{d+3}}\Bigg], \quad\theta\in (0,1),
$$
and $\lambda_0=\lambda_0(d, \Lambda)$ be the constant from Proposition \ref{200810@prop1}.
If  Assumptions \ref{A11} $(\gamma;m,M)$ and \ref{A2} $(\theta)$ are satisfied with these $\gamma$ and $\theta$, then the following assertion holds.
Let $u\in \cH^1_{2,\cD^T, \rm{loc}}(\cQ^T)$ satisfy
$$
\begin{cases}
\cP u-\lambda u = D_ig_i+f  & \text{in }\, \cQ^T,\\
\cB u = g_in_i& \text{on }\, \cN^T,\\
u = 0 & \text{on }\, \cD^T,
\end{cases}
$$
where $\lambda>0$, $\lambda\ge\lambda_0$, $g=(g_1,\ldots, g_d)\in L_{p,{\rm{loc}}}(\cQ^T)^d$, and $f\in L_{p, {\rm{loc}}}(\cQ^T)$.
Then for any $X_0\in \overline{\cQ^T}$ and $R\in (0, R_0]$,
there exist
$$
W,\,V\in L_2(\cQ_{\mu R}(X_0)),
$$
where $\mu=\frac{1}{16\cdot 96}$,
such that
$$
U\le W+V \quad \text{in }\, \cQ_{\mu  R}(X_0),
$$
\begin{equation}		\label{200805@eq5}
(W^2)_{\cQ_{\mu  R}(X_0)}^{1/2}\lesssim_{d,\Lambda,M,p} (\gamma+\theta)^{1/2-1/p_0} (U^{p_0})^{1/p_0}_{\cQ_{R/2}(X_0)}+(F^{2})^{1/2}_{\cQ_R(X_0)},
\end{equation}
and for $q\in \big[1,\frac{2(m+2)}{m+1}\big)$,
\begin{equation}		\label{200805@eq5a}
(V^q)^{1/q}_{\cQ_{\mu R}(X_0)}\lesssim_{d,\Lambda, M, p,q} (U^{p_0})^{1/p_0}_{\cQ_{R/2}(X_0)}+(F^2)^{1/2}_{\cQ_R(X_0)},
\end{equation}
where $p_0=p_0(d,\Lambda, M, p)\in (2,p)$ is from Lemma \ref{200710@lem5}.
\end{proposition}

\begin{remark}		\label{200826@rmk1}
In the proposition above,  by applying the reverse H\"older's inequality in Lemma \ref{200710@lem5} to  \eqref{200805@eq5} and \eqref{200805@eq5a}, we have that
$$
(W^2)_{\cQ_{\mu  R}(X_0)}^{1/2}\lesssim  (\gamma+\theta)^{1/2-1/p_0} (U^{2})^{1/2}_{\cQ_{R}(X_0)}+(F^{p_0})^{1/p_0}_{\cQ_R(X_0)}
$$
and
$$
(V^q)^{1/q}_{\cQ_{\mu R}(X_0)}\lesssim (U^{2})^{1/2}_{\cQ_{R}(X_0)}+(F^{p_0})^{1/p_0}_{\cQ_R(X_0)}.
$$
\end{remark}

\begin{proof}[Proof of Proposition \ref{200726@prop1}]
Based on a covering argument and Lemma \ref{200707@lem2}, it suffices to consider the following four cases:
$$
\text{(i) }\, Q_R(X_0)\subset \cQ, \quad \text{(ii) }\, X_0\in \partial \cQ, \quad Q_R(X_0)\cap \partial \cQ\subset \cD,
$$
$$
\text{(iii) }\, X_0\in \partial \cQ, \quad Q_R(X_0)\cap \partial \cQ\subset \cN, \quad \text{(iv) }\, X_0\in \Gamma.
$$
Note that the first three cases can be proved by reducing the proof of Case (iv), in which we need to deal with mixed Dirichlet-conormal boundary conditions.
Therefore, we only present here the detailed proof of the proposition with $\mu =1/96$ for  Case (iv).

By translation we may assume that $X_0=(0,0)$.
Fix the coordinate system associated with the origin and $R$ satisfying the conditions in Assumption \ref{A11} $(\gamma; m, M)$, i.e.,
$$
	\{x:\gamma R<x^1\}\cap B_{R}\subset \Omega_{R}\subset \{x:-\gamma R<x^1\}\cap B_R,
$$
$$
\big(\partial \cQ\cap Q_{R}\cap \{(t,x):x^2>\phi+\gamma R\}\big)\subset \cD,
$$
$$
\big(\partial \cQ\cap Q_{R}\cap \{(t,x): x^2<\phi-\gamma R\}\big)\subset \cN.
$$
Let $\chi=\chi(x)$ be an infinitely differentiable function defined on $\bR^d$ such that
\begin{equation}\label{eqn-201108-0139-1}
	0\le \chi\le 1, \quad |D\chi|\lesssim_d \frac{1+M}{\gamma R},
\end{equation}
$$
\chi=0 \quad \text{in }\, \{x:x^1<\gamma R, \, x^2>\phi-\gamma R\},
$$
and
\begin{equation}\label{eqn-201108-0139-2}
	\chi=1 \quad \text{in }\, \bR^d\setminus \{x:x^1<2\gamma R, \, x^2>\phi-2\gamma R\}.
\end{equation}
We define
$$
\cP_0 u =-u_t+D_i (a_0^{ij}D_j u),
$$
where $a_0^{ij}=(a^{ij})_{\cQ_R}$ is symmetric, and denote by $\cB_0$ the conormal derivative operator associated with $\cP_0$.
Observe that  $\chi u$ satisfies
\begin{equation}		\label{201107@A1}
\begin{cases}
\cP_0 (\chi u)-\lambda (\chi u) = D_i g^*_i+f^* & \text{in }\, \cQ_{R/4},\\
\cB_0 (\chi u) = g^*_in_i& \text{on }\, (-(R/4)^2,0)\times N_{R/4},\\
u = 0 & \text{on }\, (-(R/4)^2,0)\times D_{R/4},
\end{cases}
\end{equation}
where
$$
D_{R/4}= \partial \Omega \cap B_{R/4}\cap  \{x: x^2>\phi-\gamma R\},
$$
$$
N_{R/4}=\partial \Omega \cap B_{R/4}\cap  \{x: x^2<\phi-\gamma R\},
$$
\begin{equation}\label{eqn-201215-1221-1}
	f^*=a^{ij}D_j uD_i\chi -g_i D_i \chi+f\chi,
\end{equation}
\begin{equation}\label{eqn-201215-1221-2}
	g^*_i=(a^{ij}_0-a^{ij})D_j (\chi u)+a^{ij}u D_j \chi+g_i \chi.
\end{equation}
Note that the separation between $N_{R/4}$ and $D_{R/4}$ is time-independent.
We decompose
\begin{equation}		\label{200731@eq2}
\chi u=u^{(1)}+u^{(2)} \quad \text{in }\, \cQ_{R/4},
\end{equation}
where $u^{(1)}\in \cH^1_2(\cQ^0)$ is a unique weak solution of the problem
\begin{equation}		\label{200716@eq1}
\begin{cases}
\cP_0 u^{(1)}-\lambda  u^{(1)} = D_i (g^*_i\bI_{\cQ_{R/4}} )+f^*\,\bI_{\cQ_{R/4}} & \text{in }\, \cQ^0,\\
\cB_0 u^{(1)} = (g^*_i\bI_{\cQ_{R/4}})n_i& \text{on }\, (-\infty,0)\times N_{R/4},\\
u^{(1)} = 0 & \text{on }\, (-\infty,0)\times (\partial \Omega\setminus N_{R/4}).
\end{cases}
\end{equation}
Note that the existence of $u^{(1)}$ is due to Lemma \ref{200810@LEM1}, and that
 $u^{(2)}\in \cH^1_2(\cQ_{R/4})$ satisfies
\begin{equation}		\label{200731@A1}
\begin{cases}
\cP_0 u^{(2)}-\lambda u^{(2)} = 0 & \text{in }\, \cQ_{R/4},\\
\cB_0 u^{(2)} = 0 & \text{on }\, (-(R/4)^2,0)\times N_{R/4},\\
u = 0 & \text{on }\, (-(R/4)^2,0)\times D_{R/4}.
\end{cases}
\end{equation}
We divide the rest of the proof into several steps.

{\it{Step 1}}.
In this first step, we estimate
$$
U^{(1)}:=|Du^{(1)}|+\sqrt{\lambda}|u^{(1)}|
$$
and
$$
U^{(2)}:=|Du^{(2)}|+\sqrt{\lambda}|u^{(2)}|, \quad U_t^{(2)}:=|Du_t^{(2)}|+\sqrt{\lambda}|u_t^{(2)}|,
$$
where $U_t^{(2)}$ is well defined in $\cQ_{R/8}$ due to Proposition \ref{200825@prop1}.
Since the boundary portions in \eqref{200716@eq1} are time-independent, we can apply $u^{(1)}$ as a test function to get
$$
\int_\Omega |u^{(1)}(0,\cdot)|^2\,dx+\int_{\cQ^0} a^{ij}_0D_j u^{(1)} D_i u^{(1)}\,dX+\lambda \int_{\cQ^0} |u^{(1)}|^2\,dX
$$
$$
=\int_{\cQ^0}  g^*_i\bI_{\cQ_{R/4}} D_i u^{(1)}\,dX-\int_{\cQ^0} f^*\,\bI_{\cQ_{R/4}} u^{(1)}\,dX.
$$
By Young's inequality, the Poincar\'e inequality, and \eqref{eqn-201215-1221-1}-\eqref{eqn-201215-1221-2}, this implies that
$$
\|U^{(1)}\|_{L_2(\cQ^0)} \lesssim_{d,\Lambda} \|F\|_{L_2(\cQ_{R/4})}+\|Du \bI_{\operatorname{supp} D\chi}\|_{L_2(\cQ_{R/4})}
+\big\|(a_0^{ij}-a^{ij})D(\chi u)\big\|_{L_2(\cQ_{R/4})}+\|uD\chi\|_{L_2(\cQ_{R/4})}.
$$
Notice from Lemma \ref{200710@lem5} that $u$ is in $\cH^1_{p_0,\rm{loc}}(\cQ^T)$.
Hence, by H\"older's inequality and Lemma \ref{200716@lem2}, we have
$$
\|Du \bI_{\operatorname{supp} D\chi}\|_{L_2(\cQ_{R/4})}
\lesssim (\gamma R^{d+2})^{1/2-1/p_0}\|Du\|_{L_{p_0}(\cQ_{R/2})}
$$
and
\begin{align}
\|uD\chi\|_{L_2(\cQ_{R/4})}
\nonumber
&\lesssim \frac{(\gamma R^{d+2})^{1/2-1/p_0}}{\gamma R} \| u \bI_{\operatorname{supp} D\chi} \|_{L_{p_0}(\cQ_{R/4})}\\
\label{200731@eq1}
&\lesssim (\gamma R^{d+2})^{1/2-1/p_0}\|Du\|_{L_{p_0}(\cQ_{R/2})},
\end{align}
which together with Assumption \ref{A2} $(\theta)$ yields
$$
\begin{aligned}
&\big\|(a_0^{ij}-a^{ij})D(\chi u)\big\|_{L_2(\cQ_{R/4})}\\
&\lesssim \|a_0^{ij}-a^{ij}\|_{L_{2p_0/(p_0-2)}(\cQ_{R/4})} \|Du\|_{L_{p_0}(\cQ_{R/4})}+\|uD\chi\|_{L_2(\cQ_{R/4})}\\
&\lesssim (\theta+\gamma)^{1/2-1/p_0}R^{(d+2)/2}(|Du|^{p_0})^{1/p_0}_{\cQ_{R/2}}.
\end{aligned}
$$
Combining the estimates above, we reach
\begin{equation}		\label{200731@eq1b}
((U^{(1)})^2)^{1/2}_{\cQ_{R/4}}\lesssim_{d,\Lambda, M,p} (\theta+\gamma)^{1/2-1/p_0}(|Du|^{p_0})^{1/p_0}_{\cQ_{R/2}}+(F^{2})^{1/2 }_{\cQ_R}.
\end{equation}
For the estimates of $U^{(2)}$ and $U_t^{(2)}$,
we use  \eqref{200731@eq2}, \eqref{200731@eq1}, and \eqref{200731@eq1b} to obtain
\begin{align}		
\nonumber
((U^{(2)})^2)_{\cQ_{R/4}}^{1/2}&\lesssim ((U^{(1)})^2)^{1/2}_{\cQ_{R/4}}+((\chi U)^2)^{1/2}_{\cQ_{R/4}}+(|u D \chi|^2)^{1/2}_{\cQ_{R/4}}\\
\nonumber
&\lesssim (U^{p_0})^{1/p_0}_{\cQ_{R/2}}+(F^{2})^{1/2}_{\cQ_R}.
\end{align}
Therefore, by Proposition \ref{200825@prop1} we get
\begin{equation}		\label{200826@eq2}
((U^{(2)})^{p_0})^{1/p_0}_{\cQ_{R/8}}+R^2((U_t^{(2)})^{p_0})^{1/p_0}_{\cQ_{R/8}}\lesssim_{d,\Lambda, M, p} (U^{p_0})^{1/p_0}_{\cQ_{R/2}}+(F^{2})^{1/2}_{\cQ_R}.
\end{equation}

{\it{Step 2}}.
In this step, we decompose $u^{(2)}$.
Since $u^{(2)}$ satisfies \eqref{200731@A1}, $\chi u^{(2)}$ satisfies
\begin{equation}		\label{200731@eq3a}
\begin{cases}
\cP_0 (\chi u^{(2)})-\lambda (\chi u^{(2)}) = D_i h_i+h & \text{in }\, Q_{R/4}^\gamma,\\
\cB_0 (\chi u^{(2)}) =h_in_i & \text{on }\, (-(R/4)^2,0)\times N^\gamma_{R/4},\\
\chi u^{(2)} = 0 & \text{on }\, (-(R/4)^2,0)\times D^{\gamma}_{R/4},
\end{cases}
\end{equation}
where
$$
Q_{R/4}^\gamma=(-(R/4)^2, 0)\times B_{R/4}^\gamma, \quad B_{R/4}^\gamma=B_{R/4}\cap \{x:x^1>\gamma R\},
$$
$$
D^\gamma_{R/4}=B_{R/4}\cap  \{x:x^1=\gamma R, \, x^2>\phi-\gamma R\},
$$
$$
N^\gamma_{R/4}=B_{R/4}\cap  \{x:x^1=\gamma R, \, x^2<\phi-\gamma R\},
$$
$$
h(t,x)=\big[a_0^{ij}D_j u^{(2)} D_i \chi\big](t,x)+[u_t^{(2)}\chi+a_0^{ij}D_j u^{(2)}D_i \chi+\lambda u^{(2)}\chi](t,z) \bI_{\Omega^*}(x),
$$
$$
h_1(t,x)=\big[a_0^{1j}D_j\chi u^{(2)}\big](t,x)+\big[a_0^{1j}D_j u^{(2)}\chi\big](t,z)\bI_{\Omega^*}(x),
$$
and for $i\in \{2,\ldots, d\}$,
$$
h_i(t,x)=\big[a_0^{ij}D_j\chi u^{(2)}\big](t,x)-\big[a_0^{ij}D_j u^{(2)}\chi\big](t,z)\bI_{\Omega^*}(x).
$$
In the above, we denote
$$
z=\mathcal{R}x=(2\gamma R-x^1,x^2,\ldots,x^d)
$$
to be the reflection $x$ with respect to $\{x^1=\gamma R\}$ and
$$
\Omega^*=\{(x^1,x'): (2\gamma R-x^1,x')\in \Omega_{R/4}\setminus B_{R/4}^\gamma\}\cap B_{R/4}^\gamma.
$$
Indeed, one can check that \eqref{200731@eq3a} holds as follows.
Let $\varphi\in C^\infty(\overline{Q_{R/4}^\gamma})$ which vanishes on $(-(R/4)^2,0)\times\big(\partial B_{R/4}^\gamma\setminus N_{R/4}^\gamma\big)$ and $\{t=-(R/4)^2,0\}\times B_{R/4}^\gamma$.
We extend $\varphi$ to $[-(R/4)^2,0]\times \{x:x^1\ge \gamma R\}$ by setting $\varphi\equiv 0$ on $[-(R/4)^2,0]\times (\{x:x^1\ge \gamma R\}\setminus B_{R/4}^{\gamma})$, and define
$$
\tilde{\varphi}(t,x)=
\begin{cases}
\varphi(t,x) &\quad \text{if }\, x^1>\gamma R,\\
\varphi(t, 2\gamma R-x^1,x^2,\ldots, x^d) &\quad \text{otherwise}.
\end{cases}
$$
Since $\chi \tilde{\varphi}$ belongs to $C^\infty(\overline{\cQ_{R/4}})$ and vanishes on $[-(R/4)^2,0]\times (\partial \Omega_{R/4}\setminus N_{R/4})$, we can test \eqref{200731@A1} with $\chi \tilde{\varphi}$  to get
$$
\int_{\cQ_{R/4}} \big( u^{(2)}_t\chi \tilde{\varphi}+a_0^{ij}D_j u^{(2)}D_i (\chi \tilde{\varphi})\,dX+\lambda u^{(2)}\chi \tilde{\varphi}\big)\,dX=0.
$$
From this identity and the definition of $\tilde{\varphi}$, it follows that
$$
\begin{aligned}
&\int_{Q_{R/4}^\gamma} \big((\chi u^{(2)})_t{\varphi}+a_0^{ij}D_j (\chi u^{(2)})D_i {\varphi}+\lambda \chi u^{(2)}{\varphi}\big)\,dX\\
&=\int_{Q_{R/4}^\gamma} a_0^{ij}D_j \chi u^{(2)} D_i \varphi\,dX -\int_{\cQ_{R/4}\setminus Q_{R/4}^\gamma} a_0^{ij}D_j u^{(2)}\chi D_i \tilde{\varphi}\,dX\\
&\quad -\int_{Q_{R/4}^\gamma} a_0^{ij}D_j u^{(2)}D_i \chi \varphi\,dX\\
&\quad -\int_{\cQ_{R/4}\setminus Q_{R/4}^\gamma} \big(u^{(2)}_t\chi+a_0^{ij}D_j u^{(2)}D_i \chi +\lambda u^{(2)}\chi\big)\tilde{\varphi}\,dX \\
&=\int_{Q_{R/4}^\gamma}\big(h_i D_i \varphi-h\varphi\big)\,dX,
\end{aligned}
$$
which is exactly the weak formulation of \eqref{200731@eq3a}.

Now we decompose
$$
\chi u^{(2)}=u^{(3)}+u^{(4)} \quad \text{in }\, Q_{R/4}^\gamma,
$$
where $u^{(3)}\in \cH^1_2((-\infty,0)\times B_{R/4}^\gamma)$ is a unique weak solution of the problem
\begin{equation}		\label{200803@eq1}
\begin{cases}
\cP_0 u^{(3)}-\lambda u^{(3)} = D_i \big(h_i\bI_{Q_{R/16}^\gamma}\big)+h \bI_{Q_{R/16}^\gamma} & \text{in }\, Q_{R/4}^\gamma,\\
\cB_0  u^{(3)} =\big(h_i \bI_{Q^{\gamma}_{R/16}}\big)n_i & \text{on }\, (-\infty,0)\times N^\gamma_{R/4},\\
u^{(3)} = 0 & \text{on }\, (-\infty,0)\times \big(\partial B_{R/4}^\gamma\setminus N^{\gamma}_{R/4}\big),
\end{cases}
\end{equation}
and $u^{(4)}\in \cH^1_2(Q_{R/4}^\gamma)$ satisfies
\begin{equation}		\label{200804@eq5}
\begin{cases}
\cP_0 u^{(4)}-\lambda u^{(4)} = 0 & \text{in }\, Q_{R/16}^\gamma,\\
\cB_0  u^{(4)} =0  & \text{on }\, (-(R/16)^2,0)\times N^\gamma_{R/16},\\
u^{(4)} = 0 & \text{on }\, (-(R/16)^2,0)\times D_{R/16}^\gamma.
\end{cases}
\end{equation}

{\it{Step 3}}.
In this step, we estimate
$$
U^{(3)}:=\big(|Du^{(3)} |+\sqrt{\lambda}|u^{(3)}|\big)\bI_{Q_{R/4}^\gamma}.
$$
By applying  $u^{(3)}$ as a test function to \eqref{200803@eq1}, we see that
$$
\int_{Q_{R/4}^\gamma}(U^{(3)})^2\,dX\lesssim \sum_{i=1}^6 J_i,
$$
where
$$
\begin{aligned}
J_1&=\int_{Q_{R/16}^\gamma}|u^{(2)} D\chi| |Du^{(3)}|\,dX,\\
J_2&=\int_{Q_{R/16}^\gamma} \big|[D u^{(2)}\chi](t,z)\bI_{\Omega^*(x)}\big| |D u^{(3)}|\,dX,\\
J_3&=\int_{Q_{R/16}^\gamma}\big|Du^{(2)}D\chi \big||u^{(3)}|\,dX,\\
J_4&=\int_{Q_{R/16}^\gamma} \big|[u_t^{(2)}\chi](t,z)\bI_{\Omega^*}(x)\big| |u^{(3)}|\,dX,\\
J_5&=\int_{Q_{R/16}^\gamma} \big|[Du^{(2)}D\chi](t,z)\bI_{\Omega^*}(x)\big| |u^{(3)}|\,dX,\\
J_6&=\int_{Q_{R/16}^\gamma} \big|[\lambda u^{(2)}\chi ](t,z)\bI_{\Omega^*}(x)\big| |u^{(3)}|\,dX.
\end{aligned}
$$
Set
$$
\cK=\gamma^{1/2-1/p_0}R^{(d+2)/2}\big((U^{p_0})^{1/p_0}_{\cQ_{R/2}}+(F^2)^{1/2}_{\cQ_R}\big)\|U^{(3)}\|_{L_2(Q_{R/4}^\gamma)}.
$$
Similar to \eqref{200731@eq1}, we have
$$
\|u^{(2)}D\chi\|_{L_2(\cQ_{R/16})}\lesssim (\gamma R^{d+2})^{1/2-1/p_0}\|Du^{(2)}\|_{L_{p_0}(\cQ_{R/8})}.
$$
Thus by \eqref{200826@eq2} we get $J_1\lesssim \cK$.
Using H\"older's inequality and \eqref{200826@eq2}, we obtain
\begin{align*}
	J_2
	&\lesssim\|D u^{(2)}\|_{L_2(\cQ_{R/16}\setminus Q_{R/16}^\gamma)}\|Du^{(3)}\|_{L_2(Q_{R/16}^\gamma)}\\
	&\lesssim (\gamma R^{d+2})^{1/2-1/p_0}\|D u^{(2)}\|_{L_{p_0}(\cQ_{R/16}\setminus Q_{R/16}^\gamma)}\|Du^{(3)}\|_{L_2(Q_{R/16}^\gamma)}\lesssim \cK
\end{align*}
and
$$
J_6
\lesssim \sqrt{\lambda}\|u^{(2)}\|_{L_2(\cQ_{R/16}\setminus Q_{R/16}^\gamma)}\cdot \sqrt{\lambda}\|u^{(3)}\|_{L_2(Q_{R/16}^\gamma)}\lesssim \cK.
$$
Similarly, we obtain
$$
\begin{aligned}
J_4&\lesssim \|u_t^{(2)}\|_{L_2(\cQ_{R/16}
\setminus Q_{R/16}^\gamma)}\|u^{(3)}\|_{L_2(Q_{R/16}^\gamma)}\lesssim (\gamma R^{d+2})^{1/2-1/p_0}\|u_t^{(2)}\|_{L_{p_0}(\cQ_{R/16}
	\setminus Q_{R/16}^\gamma)}\|u^{(3)}\|_{L_2(Q_{R/16}^\gamma)}\\
&\lesssim (\gamma R^{d+2})^{1/2-1/p_0}R^2 \|Du_t^{(2)}\|_{L_{p_0} (\cQ_{R/8})}\|Du^{(3)}\|_{L_2(Q_{\rho/16}^\gamma)}\lesssim \cK,
\end{aligned}
$$
where we also applied Lemma \ref{200820@lem2} to $u^{(2)}_t$ and the boundary Poincar\'e inequality on half balls to $u^{(3)}$. To estimate $J_3$ and $J_5$, we note that by the same argument as in the proof of Lemma \ref{200716@lem2},
we have
$$
\| u^{(3)}\bI_{\operatorname{supp}D\chi}\|_{L_2(Q_{R/16}^\gamma)} + \|u^{(3)}\bI_{\operatorname{supp}(D\chi\circ\mathcal{R})}\|_{L_2(Q_{R/16}^\gamma)}\lesssim \gamma R \|Du^{(3)}\|_{L_2(Q_{R/8}^\gamma)}.
$$
where $\mathcal{R}$ is the reflection map: $\mathcal{R}x=(2\gamma R-x^1,x^2,\ldots, x^d)$. This together with H\"older's inequality and \eqref{200826@eq2} yields that
$$
J_3+J_5\lesssim \|Du^{(2)}\bI_{\operatorname{supp}D\chi} \|_{L_2(\cQ_{R/16})}\cdot \frac{1}{\gamma R}\big(\|u^{(3)}\bI_{\operatorname{supp}D\chi} \|_{L_2(Q_{R/16}^\gamma)}+ \|u^{(3)}\bI_{\operatorname{supp}(D\chi\circ\mathcal{R})} \|_{L_2(Q_{R/16}^\gamma)}\big)
\lesssim \cK.
$$
Collecting the estimates for $J_i$ and using Young's inequality, we conclude that
\begin{equation}\label{eqn-201215-0606}
	((U^{(3)})^2)^{1/2}_{\cQ_{R/4}}\lesssim  \gamma^{1/2-1/p_0}\big((U^{p_0})^{1/p_0}_{\cQ_{R/2}}+(F^{2})^{1/2}_{\cQ_R}\big).
\end{equation}

{\it{Step 4}}.
We are ready to complete the proof of the proposition.
From the decompositions above, we have
$$
u=w+v, \quad U\le W+V \quad \text{in }\, \cQ_{R/4},
$$
where
$$
w=(1-\chi)u+u^{(1)}+(1-\chi)u^{(2)}+\chi u^{(2)}\bI_{\cQ_{R/4}\setminus Q_{R/4}^\gamma}+u^{(3)}\bI_{Q_{R/4}^\gamma}, \quad v=u^{(4)}\bI_{Q_{R/4}^\gamma},
$$
and
$$
W=|Dw|+\sqrt{\lambda} |w|, \quad V=|Dv|+\sqrt{\lambda}|v|.
$$
Observe that
$$
W\le (1-\chi)U+|uD\chi|+U^{(1)}+(1-\chi)U^{(2)}+|u^{(2)}D\chi|+|u^{(2)}
\bI_{\cQ_{R/4}\setminus Q_{R/4}^\gamma}|+U^{(3)},
$$
where by H\"older's inequality, \eqref{200826@eq2},  and Lemma \ref{200716@lem2}, we have
$$
((1-\chi)^2U^2)^{1/2}_{\cQ_{R/4}}\lesssim \gamma^{1/2-1/p_0}(U^{p_0})^{1/p_0}_{\cQ_{R/4}},
$$
\begin{equation*}
	(|uD\chi|)^{1/2}_{\cQ_{R/4}} \lesssim \frac{1}{\gamma R} (|u\mathbb{I}_{\operatorname{supp}(D\chi)}|)^{1/2}_{\cQ_{R/4}}\lesssim \gamma^{1/2-1/p_0}(U^{p_0})^{1/p_0}_{\cQ_{R/2}},
\end{equation*}
$$
((1-\chi)^2(U^{(2)})^2)^{1/2}_{\cQ_{R/8}}\lesssim  \gamma^{1/2-1/p_0}\big((U^{p_0})^{1/p_0}_{\cQ_{R/2}}+(F^{2})^{1/2}_{\cQ_R}\big),
$$
and
$$
(|u^{(2)}D\chi|^2)^{1/2}_{\cQ_{R/16}}+(|u^{(2)}\bI_{\cQ_{R/4}\setminus \cQ_{R/4}^{\gamma}}|^2)^{1/2}_{\cQ_{R/16}}\lesssim \gamma^{1/2-1/p_0}\big((U^{p_0})^{1/p_0}_{\cQ_{R/2}}+(F^{2})^{1/2}_{\cQ_R}\big).
$$
These estimates together with \eqref{200731@eq1b} and \eqref{eqn-201215-0606} imply
\begin{equation}		\label{200827@eq2}
(W^2)^{1/2}_{\cQ_{R/16}}\lesssim (\gamma+\theta)^{1/2-1/p_0}\big((U^{p_0})^{1/p_0}_{\cQ_{R/2}}+(F^{2})^{1/2}_{\cQ_R}\big).
\end{equation}
It remains to obtain the estimate of $V$.
Observe that
$$
\begin{aligned}
B_{R/96}^\gamma &\subset B_{R/48}(y_0)\cap \{x:x^1>\gamma R\}\\
&\subset B_{R/24}(y_0)\cap \{x:x^1>\gamma R\}\subset B_{R/16}^\gamma,
\end{aligned}
$$
where $y_0=(\gamma R,-\gamma R,0,\ldots, 0)\in \bR^d$.
Thus by Proposition \ref{200826@prop1} applied to \eqref{200804@eq5}, $V\leq U+W$, and \eqref{200827@eq2}, we have
$$
(V^q)_{\cQ_{R/96}}^{1/q}\lesssim (V^2)^{1/2}_{\cQ_{R/16}}\lesssim (U^2)^{1/2}_{\cQ_{R/16}} + (W^2)^{1/2}_{\cQ_{R/16}}\lesssim (U^{p_0})^{1/p_0}_{\cQ_{R/2}}+(F^{2})^{1/2}_{\cQ_R}.
$$
The proposition is proved.
\end{proof}

\subsection{Level set estimates}
For a function $v$ on $\cQ^T$, we define its maximal function by
$$
\cM v(X)=\sup_{Q_r(Z)\ni X}\dashint_{Q_r(Z)} |v(Y)| \bI_{\cQ^T}\,dY.
$$
Let $p>2$ and  $p_0=p_0(d, \Lambda,M,p)\in (2,p)$ be from Lemma \ref{200710@lem5}.
For any $\textsf{s}>0$, denote
\begin{equation}		\label{200811@EQ1}
\begin{gathered}
\cA(\textsf{s})=\{X\in \cQ^T: (\cM U^2(X))^{1/2}>\textsf{s}\},\\
\cB(\textsf{s})=\{X\in \cQ^T: (\cM U^2(X))^{1/2}+(\gamma+\theta)^{1/p_0-1/2}(\cM F^{p_0}(X))^{1/p_0}>\textsf{s}\}.
\end{gathered}
\end{equation}

\begin{lemma}		\label{200806@lem1}
Under the same conditions of Proposition \ref{200726@prop1}, for any $q\in \big[1, \frac{2(m+2)}{m+1}\big)$, there exists a constant $C_0>0$, depending only on $d$, $\Lambda$, $M$, $p$, and $q$, such that the following holds.
For any $X_0=(t_0,x_0)\in \overline{\cQ^T}$, $R\in (0, R_0]$, $\kappa\ge 4^{(d+2)/2}$, and $\textsf{s}>0$,
if
$$
|Q_{\mu R/2}(X_0)\cap \cA(\kappa \textsf{s} )|\ge C_0\bigg(\frac{(\gamma+\theta)^{1-2/p_0}}{\kappa^2}+\frac{1}{\kappa^q}\bigg)\cdot |Q_{\mu R/2}(X_0)|,
$$
 then
$$
\cQ_{\mu R/2}(X_0)\subset \cB(\textsf{s}),
$$
where $\mu=\frac{1}{16\cdot 96}$.
\end{lemma}

\begin{proof}
By dividing the equation by $\textsf{s}$ and translating the coordinates, we may assume that $\textsf{s}=1$ and $X_0=(0,0)$.
We prove the contrapositive of the statement.
Suppose that there is a point $Y\in \cQ_{\mu R/2}$ satisfying
$$
Y\notin \cB(1).
$$
By the definition of $\cB(1)$,
\begin{equation}		\label{200811@eq5}
(\cM U^2(Y))^{1/2}+(\gamma+\theta)^{1/p_0-1/2}(\cM F^{p_0}(Y))^{1/p_0}\le 1.
\end{equation}
Set
$$
t^*_0=\min\{ (\mu R/2)^2, T\}, \quad X^*_0=(t^*_0,0)\in \overline{\cQ^T},
$$
and observe that
\begin{equation}		\label{200811@eq6}
Y\in \cQ_{\mu R/2}\subset \cQ_{\mu R}(X_0^*).
\end{equation}
From the definition of the maximal function and \eqref{200811@eq5}, we have
$$
\big(U^2\bI_{\cQ^T}\big)^{1/2}_{Q_{\mu R}(X_0^*)}+(\gamma+\theta)^{1/p_0-1/2}(F^{p_0}\bI_{\cQ^T}\big)^{1/p_0}_{Q_{\mu R}(X_0^*)}\le 1.
$$
Hence, by Proposition \ref{200726@prop1} with the estimates in Remark \ref{200826@rmk1}, there exist functions $W, V$ defined on $\cQ_{\mu R}(X_0^*)$ such that
\begin{equation}		\label{200810@eq1}
\begin{gathered}
U\le W+V \quad \text{in }\, \cQ_{\mu R}(X_0^*),\\
(W^2)^{1/2}_{\cQ_{\mu R}(X_0^*)}\lesssim_{d,\Lambda, M, p} (\gamma + \theta)^{1/2-1/p_0}, \quad (V^q)_{\cQ_{\mu R}(X_0^*)}^{1/q}\lesssim_{d,\Lambda, M, p,q}1.
\end{gathered}
\end{equation}
We now claim that for any $X\in Q_{\mu R/2}\cap \cA(\kappa)$,
\begin{equation}		\label{200810@eq2}
\big(\cM \big(W^2\bI_{\cQ_{\mu R}(X_0^*)}\big)(X)\big)^{1/2}+\big(\cM\big(V^2\bI_{\cQ_{\mu R}(X_0^*)}\big)(X)\big)^{1/2}>\kappa.
\end{equation}
By the definition of $\cA$,  we can find $Q_r(Z)$ satisfying
\begin{equation}
                                        \label{eq2.48}
X\in Q_r(Z), \quad \big(U^2\bI_{\cQ^T}\big)^{1/2}_{Q_r(Z)}>\kappa.
\end{equation}
We can always choose the time coordinate of $Z$ to be at most $T$. Furthermore, we have $r<\mu R/2$ because otherwise, from \eqref{200811@eq5}, \eqref{200811@eq6}, and the fact that
$$
Q_r(Z)\subset Q_{4r}(X_0^*),
$$
we get
$$
\begin{aligned}
\big(U^2\bI_{\cQ^T}\big)^{1/2}_{Q_r(Z)}
&\le 4^{(d+2)/2}\big(U^2\bI_{\cQ^T}\big)^{1/2}_{Q_{4r}(X_0^*)}\\
&\le 4^{(d+2)/2}(\cM U^2(Y))^{1/2}\\
&\le 4^{(d+2)/2}\le \kappa,
\end{aligned}
$$
which contradict \eqref{eq2.48}.
Since $r<\mu R/2$, we have
$$
Q_r(Z)\subset Q_{\mu R}(X_0^*),
$$
and thus by \eqref{eq2.48},
$$
\begin{aligned}
\kappa&<\big(U^2\bI_{\cQ^T}\big)^{1/2}_{Q_r(Z)}
\le \big(W^2\bI_{\cQ^T}\big)^{1/2}_{Q_r(Z)}+\big(V^2\bI_{\cQ^T}\big)^{1/2}_{Q_r(Z)}\\
&\le \big(W^2\bI_{\cQ_{\mu R}(X_0^*)}\big)^{1/2}_{Q_r(Z)}+\big(V^2\bI_{\cQ_{\mu R}(X_0^*)}\big)^{1/2}_{Q_r(Z)}\\
&\le \big(\cM \big(W^2\bI_{\cQ_{\mu R}(X_0^*)}\big)(X)\big)^{1/2}+\big(\cM\big(V^2\bI_{\cQ_{\mu R}(X_0^*)}\big)(X)\big)^{1/2},
\end{aligned}
$$
from which we get the claim.

By \eqref{200810@eq1}, \eqref{200810@eq2}, and the Hardy-Littlewood theorem, we have
$$
\begin{aligned}
&|Q_{\mu R/2}\cap \cA(\kappa)| \\
&\le  \big|\big\{X: \big(\cM \big(W^2\bI_{\cQ_{\mu R}(X_0^*)}\big)(X)\big)^{1/2}+\big(\cM\big(V^2\bI_{\cQ_{\mu R}(X_0^*)}\big)(X)\big)^{1/2}>\kappa\big\}\big|\\
&\le \big|\big\{X: \big(\cM \big(W^2\bI_{\cQ_{\mu R}(X_0^*)}\big)(X)\big)^{1/2}>\kappa/2\big\}\big|\\
&\quad + \big|\big\{X: \big(\cM \big(V^2\bI_{\cQ_{\mu R}(X_0^*)}\big)(X)\big)^{1/2}>\kappa/2\big\}\big|\\
&\le  \frac{C}{\kappa^2}\|W\|_{L_2(\cQ_{\mu R}(X_0^*))}^2+ \frac{C}{\kappa^q} \|V\|^q_{L_q(\cQ_{\mu R}(X_0^*))}\\
&\le C\bigg(\frac{(\gamma+\theta)^{1-2/p_0}}{\kappa^2}+\frac{1}{\kappa^q}\bigg)|\cQ_{\mu R}(X_0^*)|\\
&\le C\bigg(\frac{(\gamma+\theta)^{1-2/p_0}}{\kappa^2}+\frac{1}{\kappa^q}\bigg)|Q_{\mu R/2}|,
\end{aligned}
$$
where $C=C(d,\Lambda, M, p,q)$.
This completes the proof.
\end{proof}

As a consequence of the previous lemma, we have the following regularity result for $\cH^1_2$ weak solutions.

\begin{lemma}		\label{200811@lem2}
Let $p\in \big(2, \frac{2(m+2)}{m+1}\big)$ and $\lambda_0=\lambda_0(d, \Lambda)$ be from Proposition  \ref{200810@prop1}.
There exist constants $\gamma, \theta\in (0, 1)$ depending only on $d$, $\Lambda$, $M$, and $p$, such that if
Assumptions  \ref{A11} $(\gamma;m,M)$ and \ref{A2} $(\theta)$ are satisfied with these $\gamma$ and $\theta$, then the following assertions hold.
Let $u\in \cH^1_{2,\cD^T}(\cQ^T)$ satisfy
$$
\begin{cases}
\cP u-\lambda u = D_ig_i+f  & \text{in }\, \cQ^T,\\
\cB u = g_in_i& \text{on }\, \cN^T,\\
u = 0 & \text{on }\, \cD^T,
\end{cases}
$$
where $\lambda>0$, $\lambda\ge \lambda_0$, and  $g_i, f\in L_2(\cQ^T)\cap L_p(\cQ^T)$.
Then $u$ belongs to  $\cH^1_{p, \cD^T}(\cQ^T)$ and satisfies
\begin{equation}		\label{200812@eq1}
\|U\|_{L_p(\cQ^T)}\lesssim_{d,\Lambda, M, p} R_0^{(d+2)(1/p-1/2)}\|U\|_{L_2(\cQ^T)}+ \|F\|_{L_p(\cQ^T)}.
\end{equation}
Moreover, if $u$ vanishes outside $\cQ_{\gamma R_0}(X_0)$ for some $X_0\in \bR^{d+1}$, then
\begin{equation}		\label{200812@eq1a}
\|U\|_{L_p(\cQ^T)}\lesssim_{d,\Lambda, M, p} \|F\|_{L_p(\cQ^T)}.
\end{equation}
\end{lemma}

\begin{proof}
We denote
$$
q=\frac{1}{2}\bigg(p+\frac{2(m+2)}{m+1}\bigg)\in \bigg(p,\frac{2(m+2)}{m+1}\bigg).
$$
Let $\gamma$, $\theta$, and $\kappa$ be positive constants to be chosen later, such that
$$
\gamma \in \Bigg(0, \frac{1}{160\sqrt{d+3}}\Bigg],\quad  \theta\in (0,1), \quad \kappa\ge 4^{(d+2)/2},
$$
and that
$$
\varepsilon:=C_0\bigg(\frac{(\gamma+\theta)^{1-2/p_0}}{\kappa^2}+\frac{1}{\kappa^q}\bigg)<1,
$$
where $C_0=C_0(d,\Lambda, M, p,q)=C_0(d,\Lambda, M, p)$ is the constant from Lemma \ref{200806@lem1}.
Recall the notation  \eqref{200811@EQ1}.
By the Hardy-Littlewood theorem, we have
\begin{equation}		\label{200811@eq1}
|\cA(\kappa \mathsf{s})|\le \frac{C_1(d)}{(\kappa \mathsf{s})^2}\|U\|_{L_2(\cQ^T)}^2.
\end{equation}
Using this, Lemma \ref{200806@lem1}, and the ``crawling of ink spots" lemma in \cite{MR0579490,MR0563790}, we get
$$
|\cA(\kappa \mathsf{s})|\le C_2(d)\varepsilon|\cB(\mathsf{s})|
$$
for any $\mathsf{s}>\mathsf{s}_0$, where
$$
\mathsf{s}_0^2=\frac{C_1}{\varepsilon\kappa^2 |Q_{\mu R_0/2}|}\|U\|_{L_2(\cQ^T)}^2.
$$
Hence, for a sufficiently large $\mathsf{S}>\mathsf{s}_0$,
$$
\begin{aligned}
&\int_0^{\kappa \mathsf{S}} |\cA(\mathsf{s})|\mathsf{s}^{p-1}\,d\mathsf{s}=\kappa^p\int^{\mathsf{S}}_0 |\cA(\kappa \mathsf{s})|\mathsf{s}^{p-1}\,d\mathsf{s}\\
&\le \kappa^p\int_0^{\mathsf{s}_0}|\cA(\kappa \mathsf{s})|\mathsf{s}^{p-1}\,d\mathsf{s}+C_2\varepsilon\kappa^p \int_{\mathsf{s}_0}^{\mathsf{S}} |\cB(\mathsf{s})|\mathsf{s}^{p-1}\,d\mathsf{s}\\
&=:I_1+I_2.
\end{aligned}
$$
By \eqref{200811@eq1} we have
$$
I_1\lesssim_{d,p}\varepsilon^{1-p/2}R_0^{(d+2)(1-p/2)}\|U\|_{L_2(\cQ^T)}^p,
$$
To estimate $I_2$, observe that
$$
\cB(\mathsf{s})\subset \cA(\mathsf{s}/2)\cup \big\{X\in \cQ^T:(\gamma+\theta)^{1/p_0-1/2}\big(\cM F^{p_0}(X)\big)^{1/p_0}>\mathsf{s}/2\big\},
$$
from which together with the Hardy-Littlewood maximal function theorem, we obtain
$$
\begin{aligned}
I_2
&\le C_2\varepsilon \kappa^p \int_{\mathsf{s}_0}^{\mathsf{S}}|\cA(\mathsf{s}/2)|\mathsf{s}^{p-1}\,d\mathsf{s}+C \varepsilon \kappa^p (\gamma+\theta)^{p(1/p_0-1/2)}\|(\cM F^{p_0})^{1/p_0}\|_{L_p(\cQ^T)}^p\\
&\le C_*\varepsilon \kappa^p \int_0^{\kappa\mathsf{S}}|\cA(\mathsf{s})|\mathsf{s}^{p-1}\,d\mathsf{s}+C_* \varepsilon \kappa^p (\gamma+\theta)^{p(1/p_0-1/2)}\|F\|_{L_p(\cQ^T)}^p,
\end{aligned}
$$
where $C_*=C_*(d,p)$.
Note that by \eqref{200811@eq1},
$$
\int_0^{\kappa\mathsf{S}}|\cA(\mathsf{s})|\mathsf{s}^{p-1}\,d\mathsf{s}<\infty.
$$
Thus by taking $\kappa$ sufficiently large, and then, choosing $\theta$ and $\gamma$ sufficiently small such that
$$
C_*\varepsilon \kappa^p=C_*C_0\bigg(\kappa^{p-2}(\gamma+\theta)^{1-2/p_0}+\kappa^{p-q}\bigg)\le \frac{1}{2},
$$
we conclude that for any $\mathsf{S}>0$,
$$
\int_0^{\kappa \mathsf{S}} |\cA(\mathsf{s})|\mathsf{s}^{p-1}\,d\mathsf{s}\lesssim_{d,\Lambda, M, p}R_0^{(d+2)(1-p/2)}\|U\|_{L_2(\cQ^T)}^p+\|F\|_{L_p(\cQ^T)}^p.
$$
This shows \eqref{200812@eq1}.
If we further assume that $u$ vanishes outside $\cQ_{\gamma R_0}(X_0)$, then by using \eqref{200812@eq1}, H\"older's inequality, and choosing again $\gamma$ sufficiently small, we conclude \eqref{200812@eq1a}.
The lemma is proved.
\end{proof}

%========================================
\subsection{Proof of Theorem \ref{MT1}}		
%========================================

Thanks to Proposition \ref{200810@prop1} and a duality argument,
it suffices to consider the case when $p\in \big(2,\frac{2(m+2)}{m+1}\big)$.

For the a priori estimate \eqref{200812@eq2} in the assertion $(a)$, by moving all the lower-order terms to the right-hand side of the equation, we may assume that the lower-order coefficients of $\cP$ are all zero.
Then the a priori estimate follows from \eqref{200812@eq1a} and the standard partition of unity argument.
For the solvability in the assertion $(a)$, thanks to the a priori estimate and the method of continuity, we only need to consider the case when the lower-order terms are all zero, which follows from the regularity result in Lemma \ref{200811@lem2} and the a priori estimate along with the standard approximation argument.
Finally, the proof of the assertion $(b)$ is by considering the equation of $ue^{-\lambda_0 t}$; cf.  \cite[Theorem 8.2 (iii)]{MR2835999}.
\qed

\appendix
%========================================
\section{}		\label{appendix}
%========================================

\begin{proof}[Proof of Lemma \ref{200820@lem2}]
We  may assume that $X_0=(0,0)$ and $u$ is smooth with respect to $t$.
By scaling, without loss of generality, we can also assume that $R=1$.

We first prove \eqref{eqn-200919-1145}.
Take a function $\zeta\in C^\infty_0(\Omega_{3/2})$ such that
$$
0\le \zeta\le 1,  \quad 1\lesssim \int_\Omega \zeta\,dz, \quad |D\zeta|\lesssim 1.
$$
Set
$$
v(t)=\int_{\Omega}\zeta u(t,\cdot)\,dx \bigg/\int_\Omega \zeta\,dx, \quad c=\int^0_{-(3/2)^2} v(t)\,dt.
$$
By \cite[Lemmas 5.3 and 5.4]{arXiv:2007.01986} with scaling applied to $u-c$, we have
\begin{equation}		\label{201106@eq1}
\begin{aligned}
\|u-(u)_{\cQ_1}\|_{L_{q_0, p_0}(\cQ_1)}
&\le2\|u-c\|_{L_{q_0,p_0}(\cQ_1)}
\lesssim_{d,p,q,p_0,q_0} \|u-c\|_{L_{q,p}(\cQ_{3/2})}\\
&\quad +\big(\|Du\|_{L_{q,p}(\cQ_{3/2})}
+\|g\|_{L_{q,p}(\cQ_{3/2})}\big),
\end{aligned}
\end{equation}
where, by the triangle inequality, we obtain
\begin{equation}		\label{201106@eq1a}
\|u-c\|_{L_{q,p}(\cQ_{3/2})}\le \|u-v\|_{L_{q,p}(\cQ_{3/2})}+\|v-c\|_{L_{q,p}(\cQ_{3/2})}=:J_1+J_2.
\end{equation}
Note that
$$
\|u(t,\cdot)-v(t)\|_{L_p(\Omega_{3/2})}\lesssim \|u(t,\cdot)-(u(t,\cdot))_{\Omega_{3/2}}\|_{L_p(\Omega_{3/2})}\lesssim \|Du(t,\cdot)\|_{L_p(\Omega_{2})},
$$
where we used Lemma \ref{200708@lem2} and the interior Poincar\'e inequality with a covering argument in the second inequality.
This implies that
$$
J_1\lesssim \|Du\|_{L_{q,p}(\cQ_2)}.
$$
By the Sobolev-Poincar\'e inequality in the $t$ variable, we get
\begin{equation*}
J_2\lesssim \int_{-(3/2)^2}^0 \bigg|\int_\Omega \zeta u_t\,dx\bigg|\,dt.
\end{equation*}
Since $u_t=D_i g_i$
and $\zeta$ has compact support in $\Omega_{3/2}$, integrating by parts with respect to $x$ and using H\"older's inequality,  we see that
$$
J_2\lesssim \|g\|_{L_{q,p}(\cQ_{2})}.
$$
Combining \eqref{201106@eq1}, \eqref{201106@eq1a}, and the estimates of $J_i$, we conclude \eqref{eqn-200919-1145}.

We next prove \eqref{eqn-200920-0425}.
Let $Y_0=(s_0,y_0)$. It is easily seen that the estimates above still holds with
$$
c=\int^{s_0}_{s_0-\alpha^2} v(t)\,dt.
$$
Thus, we have
$$
\|u\|_{L_{q_0, p_0}(\cQ_1)}	\lesssim_{d,p,q,p_0,q_0} \|Du\|_{L_{q,p}(\cQ_{2})}+\|g\|_{L_{q,p}(\cQ_{2})}+c.
$$
Therefore, it suffices to show that
$$
c\lesssim \|u\|_{L_1((s_0-\alpha^2,s_0)\times \Omega_1)}
\lesssim  \|Du\|_{L_{q,p}(\cQ_{2})},
$$
which follows from Lemma \ref{200508@lem1} and H\"older's inequality.
The lemma is proved.
\end{proof}

In the lemmas below, we set
$$
\Gamma(t)=\{x\in \partial \Omega: (t,x)\in \Gamma\}.
$$
Similarly, we define $\cD(t)$ and $\cN(t)$.

\begin{lemma}		\label{211112@lem1}
For $\gamma>0$, the condition $(b)$ in
Assumption \ref{A11} $(\gamma;m,M)$ implies Assumption \ref{ass-0301-2356}.
\end{lemma}

\begin{proof}
Let $R\in (0, R_0]$ and $t_0\in \bR$.
It suffices to show that, under the condition $(b)$, there exist decompositions
$$
\partial \Omega= D^{t_0}\cup N^{t_0}
$$
such that
$$
\cD(t)\subset D^{t_0}, \quad H^d(\cD(t), D^{t_0})\le 2\gamma R, \quad \forall t\in [t_0-R^2, t_0).
$$
Indeed, this follows by defining $D^{t_0}$ as the set of all points $x=(x^1, \ldots, x^d)\in \partial \Omega$ satisfying either $x\in \cD(t_0)$
or
$$
\operatorname{dist}(x, \Gamma(t_0))<R \quad \text{and}\quad
x^2>\phi(x^3,\ldots, x^{m+2})-\gamma R
$$
in the coordinate system associated with $((t_0, x_0), R)$, where $x_0\in \Gamma(t_0)$ satisfies
$\operatorname{dist}(x, x_0)=\operatorname{dist}(x, \Gamma(t_0))$.
\end{proof}

\begin{lemma}		\label{200707@lem2}
Suppose that the condition $(b)$ in Assumption \ref{A11} $(\gamma;m,M)$ holds with $\gamma\in [0,1)$.
\begin{enumerate}[$(i)$]
\item
Let $X_0=(t_0,x_0)\in \Gamma$ and $R\in (0, R_0]$.
Then for $t\in (t_0-R^2,t_0+R^2)$,
$$
B_R(x_0)\cap \Gamma(t)\neq \emptyset.
$$
\item
Let $X_0=(t_0,x_0)\in \overline{\cQ}$ and $R\in (0, R_0]$ with
$$
B_R(x_0)\cap \Gamma(t_0)= \emptyset.
$$
Then
$$
Q_{R/2}(X_0)\cap \Gamma=\emptyset.
$$
\end{enumerate}
\end{lemma}

\begin{proof}
We only prove the assertion $(i)$ because $(ii)$ is its easy consequence.
Suppose that there exists $t\in (t_0-R^2,t_0+R^2)$ such that
$$
B_R(x_0)\cap \Gamma(t)=\emptyset.
$$
Then either
$$
\{t\}\times (B_{R}(x_0)\cap \partial \Omega)\subset \cD(t) \quad \text{or}\quad \{t\}\times (B_{R}(x_0)\cap \partial \Omega)\subset \cN(t),
$$
which contradicts with the fact that
$$
\{t\}\times \big(B_{R}(x_0)\cap \partial \Omega\cap \{x: x^2>\phi+\gamma R\}\big)\subset \cD(t),
$$
$$
\{t\}\times \big(B_{R}(x_0)\cap \partial \Omega\cap \{x:x^2<\phi-\gamma R\}\big)\subset \cN(t),
$$
and
$$
\phi(x_0^3,\ldots, x_0^d)=x_0^2.
$$
in the coordinate system associated with $(X_0, R)$.
The assertion $(a)$ is proved.
\end{proof}


\begin{thebibliography}{10}

\bibitem{MR3315670}
Pascal Auscher, Nadine Badr, Robert Haller-Dintelmann, and Joachim Rehberg.
\newblock The square root problem for second-order, divergence form operators with mixed boundary conditions on $L^p$.
\newblock {\em J. Evol. Equ.} 15(1):165--208, 2015.

\bibitem{MR3170211}
Kevin Brewster, Dorina Mitrea, Irina Mitrea, and Marius Mitrea.
\newblock Extending Sobolev functions with partially vanishing traces from locally $(\varepsilon,\delta)$-domains and applications to mixed boundary problems.
\newblock {\em J. Funct. Anal.} 266(7):4314--4421, 2014.

\bibitem{MR1284808}
Russell Brown.
\newblock The mixed problem for {L}aplace's equation in a class of {L}ipschitz
  domains.
\newblock {\em Comm. Partial Differential Equations}, 19(7-8):1217--1233, 1994.

\bibitem{BC20}
	R.~M. Brown and L.~D. Croyle.
	\newblock Estimates for the $L^q$-mixed problem in $C^{1,1}$-domains.
	\newblock {\em Complex Variables and Elliptic Equations}, 0(0):1--13, 2020.

\bibitem{MR1486629}
L.~A. Caffarelli and I.~Peral.
\newblock On {$W^{1,p}$} estimates for elliptic equations in divergence form.
\newblock {\em Comm. Pure Appl. Math.}, 51(1):1--21, 1998.

\bibitem{MR3809039}
Jongkeun Choi, Hongjie Dong, and Doyoon Kim.
\newblock Conormal derivative problems for stationary {S}tokes system in
  {S}obolev spaces.
\newblock {\em Discrete Contin. Dyn. Syst.}, 38(5):2349--2374, 2018.

\bibitem{arXiv:1904.00545}
Jongkeun Choi, Hongjie Dong, and Zongyuan Li.
\newblock Optimal regularity for a {D}irichlet-conormal problem in {R}eifenberg
  flat domain.
  \newblock {\em Appl. Math. Optim.},  83(3):1547--1583, 2021.


\bibitem{MR2835999}
Hongjie Dong and Doyoon Kim.
\newblock Higher order elliptic and parabolic systems with variably partially
  {BMO} coefficients in regular and irregular domains.
\newblock {\em J. Funct. Anal.}, 261(11):3279--3327, 2011.

\bibitem{MR3013054}
Hongjie Dong and Doyoon Kim.
\newblock {\em The conormal derivative problem for higher order elliptic
  systems with irregular coefficients}, volume 581 of {\em Contemp. Math.}
\newblock Amer. Math. Soc., Providence, RI, 2012.


\bibitem{arXiv:2003.10980}
Hongjie Dong and Zongyuan Li.
\newblock The {D}irichlet-conormal problem with homogeneous and inhomogeneous
  boundary conditions.
\newblock {\em Comm. Partial Differential Equations}, 46(3):470--497, 2021.

\bibitem{MR3582221}
Karoline Disser, A. F. M. ter Elst, and Joachim Rehberg,
\newblock On maximal parabolic regularity for non-autonomous parabolic operators.
\newblock {\em J. Differential Equations} 262(3):2039--2072, 2017.

\bibitem{MR3797618}
Moritz Egert.
\newblock $L^p$-estimates for the square root of elliptic systems with mixed boundary conditions.
\newblock {\em J. Differential Equations}, 265(4):1279--1323, 2018.


\bibitem{MR2597943}
{L}awrence~{C}raig Evans.
\newblock {\em Partial differential equations}, volume~19 of {\em Graduate
  Studies in Mathematics}.
\newblock American Mathematical Society, Providence, RI, second edition, 2010.

\bibitem{MR3659370}
Stephan Fackler.
\newblock Non-autonomous maximal regularity for forms given by elliptic operators of bounded variation.
\newblock {\em  J. Differential Equations} 263(6):3533--3549, 2017.

\bibitem{MR0717034}
Mariano Giaquinta.
\newblock {\em Multiple integrals in the calculus of variations and nonlinear
  elliptic systems}, volume 105 of {\em Annals of Mathematics Studies}.
\newblock Princeton University Press, Princeton, NJ, 1983.

\bibitem{MR0990595}
Konrad Gr\"oger.
\newblock A {$W^{1,p}$}-estimate for solutions to mixed boundary value problems
  for second order elliptic differential equations.
\newblock {\em Math. Ann.}, 283(4):679--687, 1989.

\bibitem{MR3573649}
Robert Haller-Dintelmann, Alf Jonsson, Dorothee Knees, and Joachim Rehberg.
\newblock Elliptic and parabolic regularity for second-order divergence
  operators with mixed boundary conditions.
\newblock {\em Math. Methods Appl. Sci.}, 39(17):5007--5026, 2016.

\bibitem{MR2541414}
Robert Haller-Dintelmann and Joachim Rehberg.
\newblock Maximal parabolic regularity for divergence operators including mixed boundary conditions. \newblock {\em J. Differential Equations}, 247(5):1354--1396, 2009.

\bibitem{MR2810833}
Roland Herzog, Christian Meyer, and Gerd Wachsmuth.
\newblock Integrability of displacement and stresses in linear and nonlinear
  elasticity with mixed boundary conditions.
\newblock {\em J. Math. Anal. Appl.}, 382(2), 2011.

\bibitem{MR2403322}
Matthias Hieber and Joachim Rehberg.
\newblock Quasilinear parabolic systems with mixed boundary conditions on nonsmooth domains.
\newblock {\em SIAM J. Math. Anal.}, 40(1):292--305, 2008.

\bibitem{MR3190265}
B.~Tomas Johansson and Vladimir~A. Kozlov.
\newblock Solvability and asymptotics of the heat equation with mixed variable
  lateral conditions and applications in the opening of the exocytotic fusion
  pore in cells.
\newblock {\em IMA J. Appl. Math.}, 79(2):377--392, 2014.


\bibitem{arXiv:2007.01986}
Doyoon Kim, Seungjin Ryu, and Kwan Woo.
\newblock Parabolic equations with unbounded lower-order coefficients in
  sobolev spaces with mixed norms.
\newblock {\em arXiv:2007.01986}.

\bibitem{MR3804727}
Tujin Kim and Daomin Cao.
\newblock Existence of solution to parabolic equations with mixed boundary condition on non-cylindrical domains.
\newblock {\em J. Differential Equations}, 265(6):2648--2670, 2018.


\bibitem{MR0563790}
Nicolai~V. Krylov and Mikhail~V. Safonov.
\newblock A property of the solutions of parabolic equations with measurable
  coefficients.
\newblock {\em Izv. Akad. Nauk SSSR Ser. Mat.}, 44(1):161--175, 1980.

\bibitem{MR0241822}
{O}l'ga~Aleksandrovna Lady{\v{z}}enskaja, {V}sevolod~{A}lekseevich Solonnikov,
  and {N}ina~{N}ikolaevna Ural'ceva.
\newblock {\em Linear and quasilinear equations of parabolic type}.
\newblock Translated from the Russian by S. Smith. Translations of Mathematical
  Monographs, Vol. 23. American Mathematical Society, Providence, R.I., 1968.

\bibitem{MR1804512}
C.~Lederman, J.~L. V\'{a}zquez, and N.~Wolanski.
\newblock Uniqueness of solution to a free boundary problem from combustion.
\newblock {\em Trans. Amer. Math. Soc.}, 353(2):655--692, 2001.

\bibitem{MR1799414}
C.~Lederman, N.~ Wolanski, and J.~L. Vazquez.
\newblock A mixed semilinear parabolic problem in a noncylindrical space-time
  domain.
\newblock {\em Differential Integral Equations}, 14(4):385--404, 2001.

\bibitem{MR0826642}
Gary M. Lieberman.
\newblock Mixed boundary value problems for elliptic and parabolic differential equations of second order.
\newblock {\em J. Math. Anal. Appl.}, 113(2):422--440, 1986.


\bibitem{MR0579490}
Mikhail~V. Safonov.
\newblock Harnack's inequality for elliptic equations and {H}\"older property
  of their solutions.
\newblock {\em Zap. Nauchn. Sem. Leningrad. Otdel. Mat. Inst. Steklov. (LOMI)},
  96:272--287, 312, 1980.

\bibitem{MR1452171}
Giuseppe Savar\'{e}.
\newblock Regularity and perturbation results for mixed second order elliptic
  problems.
\newblock {\em Comm. Partial Differential Equations}, 22(5-6):869--899, 1997.

\bibitem{MR1444765}
Giuseppe Savar\'{e}.
\newblock Parabolic problems with mixed variable lateral conditions: an abstract approach.
\newblock {\em J. Math. Pures Appl. (9)}, 76(4):321--351, 1997.


\bibitem{MR0239272}
Eliahu Shamir.
\newblock Regularization of mixed second-order elliptic problems.
\newblock {\em Israel J. Math.}, 6:150--168, 1968.

\bibitem{MR1950984}
A. L. Skubachevskii and R. V. Shamin.
\newblock The mixed boundary value problem for parabolic differential-difference equation.
\newblock International Conference on Differential and Functional Differential Equations (Moscow, 1999).
{\em Funct. Differ. Equ.} 8(3-4):407--424, 2001.



\bibitem{MR3034453}
J.~L. Taylor, K.~A. Ott, and R.~M. Brown.
\newblock The mixed problem in {L}ipschitz domains with general decompositions
  of the boundary.
\newblock {\em Trans. Amer. Math. Soc.}, 365(6):2895--2930, 2013.

\end{thebibliography}
\end{document}